\documentclass[a4paper,11pt]{amsart}
\usepackage{amsmath,inputenc,euscript,amssymb,geometry}
\usepackage{graphicx}
\usepackage{amssymb}
\usepackage{latexsym}
\usepackage{amssymb,amsbsy,amsmath,amsfonts,amssymb,amscd}
\usepackage{hyperref}
\usepackage{graphics}
\usepackage{color}

\newcommand{\D}{\displaystyle}
\newtheorem{lemma}{Lemma}[section]
\newtheorem{theorem}[lemma]{Theorem}
\newtheorem{prop}[lemma]{Proposition}

\begin{document}
\title[]{The initial value problem for the Binormal Flow with rough data}
\author[V. Banica]{Valeria Banica}
\address[V. Banica]{Laboratoire Analyse et probabilit\'es (EA 2172)\\D\'eptartement de Math\'ematiques\\ Universit\'e d'Evry, 23 Bd. de France, 91037 Evry\\ France, Valeria.Banica@univ-evry.fr} 

\author[L. Vega]{Luis Vega}
\address[L. Vega]{Departamento de Matem\'aticas, Universidad del Pais Vasco, Aptdo. 644, 48080 Bilbao, Spain, luis.vega@ehu.es
\hfill\break\indent \and
BCAM Alameda Mazarredo 14, 48009 Bilbao, Spain, lvega@bcamath.org} 
\begin{abstract}In this article we consider the initial value problem  of the binormal flow with initial data given by curves that are regular except at one point where they have a corner. We prove that under suitable conditions on the initial data a unique regular solution exists  for strictly positive and strictly negative times. Moreover, this solution satisfies a weak version of the equation for all times and can be seen as a perturbation of a suitably chosen self-similar solution. Conversely, we also prove that if  at time $t=1$ a small regular perturbation of a self-similar solution is taken as initial condition then there exists a unique solution that at time  $t=0$ is regular except at a  point where it has  a corner  with the same angle as the one of the self-similar solution. This solution can be extended  for negative times. The proof uses the full strength of the previous papers \cite{GRV}, \cite{BV1}, \cite{BV2} and \cite{BV3} on the study of small perturbations of self-similar solutions. A compactness argument is used to avoid the weighted conditions we needed in \cite{BV3}, as well as a more refined analysis of the asymptotic in time and in space of the tangent and normal vectors. \\

\begin{center}
{\normalsize \'Evolution par le flot binormal  de courbes \`a un coin}
\end{center}

\noindent
{\tiny R\'ESUM\'E}. 
Dans cet article on consid\`ere le flot binormal avec donn\'ees initiales des courbes r\'eguli\`eres partout sauf en un point o\`u elles ont un coin. On montre sous des conditions appropri\'ees sur la donn\'ee initiale qu'il existe une unique solution r\'eguli\`ere pour des temps strictement postifs et n\'egatifs. De plus, cette solution satisfait le flot binormal en un sens faible et peut \^etre vue comme une perturbation d'une solution auto-similaire bien choisie. R\'eciproquement, on montre aussi que si \`a temps $t=1$ on prend comme donn\'ee initiale une petite perturbation r\'eguli\`ere d'une solution auto-similaire, alors il existe une unique solution, qui \`a temps $t=0$ est r\'eguli\`ere partout sauf en un point o\`u elle a un coin de m\^eme angle que celui form\'e par la solution auto-similaire. Cette solution peut \^etre prolong\'ee aux temps n\'egatifs. La preuve s'appuie sur les r\'esultats des articles pr\'ec\'edents \cite{GRV}, \cite{BV1}, \cite{BV2} et \cite{BV3} sur l'\'etude des petites perturbations des solutions auto-similaires. Un argument de compacit\'e est utilis\'e pour \'eviter les conditions \`a poids impos\'ees dans \cite{BV3}, ainsi qu'une analyse plus raffin\'ee des asymptotiques en temps et en espace des vecteurs tangent et normaux.
\end{abstract}
\maketitle

\tableofcontents

\section{Introduction}
We consider the binormal flow equation
\begin{equation}\label{binormal}
\chi_t=\chi_x\land\chi_{xx},
\end{equation}
which is a geometric law for the evolution in time of a curve $\chi(t)$ in $\mathbb R^3$, parametrized by arclength $x$. This model has been proposed in 1906 by Da Rios \cite{DaR}, and  rediscovered in 1965 by Arms and Hama \cite{ArHa}, as a model for the evolution of a vortex filament in a 3-D inhomogeneous inviscid fluid (see also \cite{Ri1},\cite{Ri2} for the history of this equation). It was also used as a model for vortex filament dynamics in superfluids (\cite{LD},\cite{LRT},\cite{Bu}). From \eqref{binormal} it follows that the tangent vector $T(t,x)$ satisfies the Schr\"odinger map equation on the sphere $\mathbb S^2$,
\begin{equation}\label{schmap}
T_t=T\land T_{xx}.
\end{equation}
Also using the Frenet equations for the tangent $T$, the normal $n$, and the binormal $b$, equation \eqref{binormal} can be written as
\begin{equation}
\label{binormal1}
\chi_t=cb,
\end{equation}
with $c(t,x)$ denoting the curvature.
Finally, Hasimoto \cite{Ha} showed that if the curvature $c(t,x)$ does not vanish, then the function
\begin{equation}
\label{eq1.2}
\psi(t,x)=c(t,x)e^{i\int_0^x \tau(t,s)\,ds},
\end{equation}
that he calls the filament function, solves the focusing cubic non-linear Schr\"odinger equation (NLS)
\begin{equation}\label{Hasimoto}
i\psi_t+\psi_{xx}+\frac \psi 2\left(|\psi|^2-A(t)\right)=0
\end{equation}
for some real function $A(t)$ that depends on $c(0,t)$ and $\tau(0,t)$. Here $\tau$ stands for the torsion. The non-vanishing constrain on the curvature has been removed by Koiso \cite{Ko}, by using another frame instead of Frenet's one.

In view of this link with the nonlinear Schr\"odinger equation, existence results were given for the initial value problem of the binormal flow with initial data curves with curvature and torsion in high order Sobolev spaces (\cite{Ha},\cite{Ko},\cite{FuMi}). The case of less regular closed curves was considered recently by Jerrard and Smets by using  a weak version of the binormal flow (\cite{JeSm1},\cite{JeSm2}).  Let us mention also that stability of various types of particular solutions of the binormal flow is a subject of current research (see for instance \cite{Cal}, \cite{Laf} and the references therein). Also, to emphasise the great complexity of the binormal flow, we recall that in the case of closed curves, various aspects of evolutions of knoted vortices by the binormal flow are studied using geometric and topological methods (as an example, see \cite{Mag} and the references therein).

We are interested in solutions of \eqref{binormal} that at a given time are regular except at a point  where they have a corner. One can use the invariance of the equation under translations in time and in space and assume without loss of generality that the time is $t=0$ and the corner is located at the origin  $(0,0,0)$. Let us also note that the equation is  reversible in time. This is because if $\chi(t,x)$ is a solution so is $\chi(-t,-x)$.
 
One relevant class of solutions  are the self-similar ones, i.e. those that can be written as
$$\chi(t,x)=\sqrt{t}\,G\left(\frac x{\sqrt{t}}\right)$$ 
for some appropriate $G$.
These solutions have been investigated first by physicists in the 80's. In fact it is rather easy to see that, modulo rotations,  self-similar solutions are a family of curves $\chi_a$ parametrized by $a\in\mathbb R^{+*}$, such that curvature and torsion of $\chi_a(t)$ at $x$ are $\frac{a}{\sqrt{t}}$ and $\frac{x}{2t}$ respectively (\cite{LD},\cite{LRT},\cite{Bu}). From this it is not complicated to conclude that $\chi_a(0)$ has a corner at  $(0,0,0)$. This fact, together with a characterization and detailed asymptotic of the self-similar solutions was proved  in \cite{GRV}. We reformulate part of Theorem 1 of \cite{GRV} as follows. Details will be given in the next section. 

\begin{theorem}\label{sssols}(Description of self-similar solutions \cite{GRV}) Let $A^+$ and $A^-$ be any two distinct non-colinear unitary vectors in $\mathbb R^3$. 
Then, there exists a unique self-similar solution $\chi$ for positive times with initial data at time $t=0$ 
$$\chi(0,x)=\left\{\begin{array}{cc}A^+x,\,x\geq 0,\\A^-x,\, x\leq 0.\end{array}\right.$$
All self-similar solutions are described in this way.
Moreover, if we denote $a\in\mathbb R^{+*}$ such that
$\sin\frac{(\widehat{A^+,-A^-})} 2=e^{-\pi\frac{a^2}{2}}$, $\frac{a}{\sqrt{t}}$ and $\frac{x}{2t}$ are respectively the curvature and the torsion of the curve $\chi(t,x)$. Also, there exist two complex vectors $B^\pm$ orthonormal to $A^\pm$ such that 
$$B^{\pm}=\underset{x\rightarrow \pm\infty}{\lim}(n+ib)(t,x)e^{i\int_0^x\tau(t,s)ds}e^{-ia^2\log\sqrt{t}+ia^2\log|x|}.$$
\end{theorem}
Up to a rotation, the coordinates of $A^\pm$ and $B^\pm$ are given explicitly in terms of Gamma functions involving the parameter $a$ (see formula (55), (57), (47), (48), and (69) in \cite{GRV}). 

\begin{figure}
    \includegraphics[width=7cm,height=6cm]{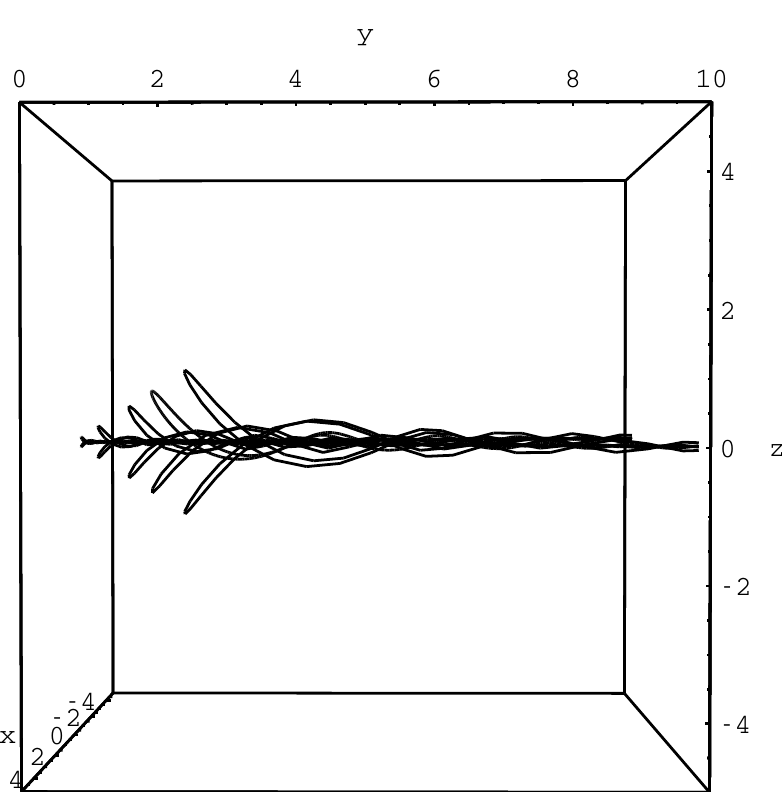} ,\includegraphics[width=7cm,height=6cm]{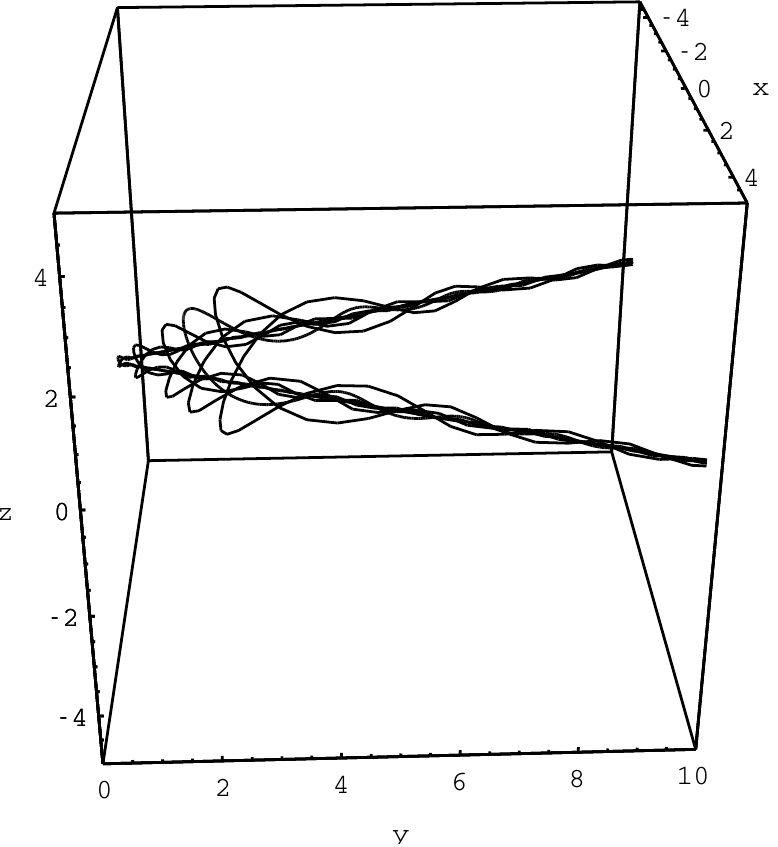} 
  \caption{Self-similar solution for negative and positive times, from two different angles.}
    \label{self-sim}
\end{figure} 

Finally, we want to remark that the solution given in the above theorem can be continued  as a self-similar solution in a unique way  for negative times. This is done as follows. If $\chi$ in the Theorem \ref{sssols} exists for negative times, then $\chi^*(t,x)=\chi(-t,-x)$ with $t>0$ is a solution for positive times with initial data $\chi(0,-x)$. In view of Theorem \ref{sssols} it follows that $\chi^*$ is unique, and it is obtained by a rotation of $\chi$ around the axis given by  the vector $A^+-A^-$ and with angle $\pi$. This rotation can be also  seen as a composition of a reflection with respect to the plane generated by $A^+$ and $A^-$, and a change of the sense of parametrization -see Figure \ref{self-sim}.

Numerical simulations for self-similar solutions have been done by Buttke in \cite{Bu} and by de la Hoz, Garc\'{i}a-Cervera and Vega in \cite{HGCV}, where a similarity at the qualitative level with the flow across a delta wing is emphasized -see Figure \ref{delta wing}.

Last but not the least we mention that the binormal flow and its self-similar solution are used to understand the architecture of the myocardium, as shown by Peskin and Mc.Queen in \cite{PMQ}. \\

\begin{figure}
    \includegraphics[width=16cm,height=6cm]{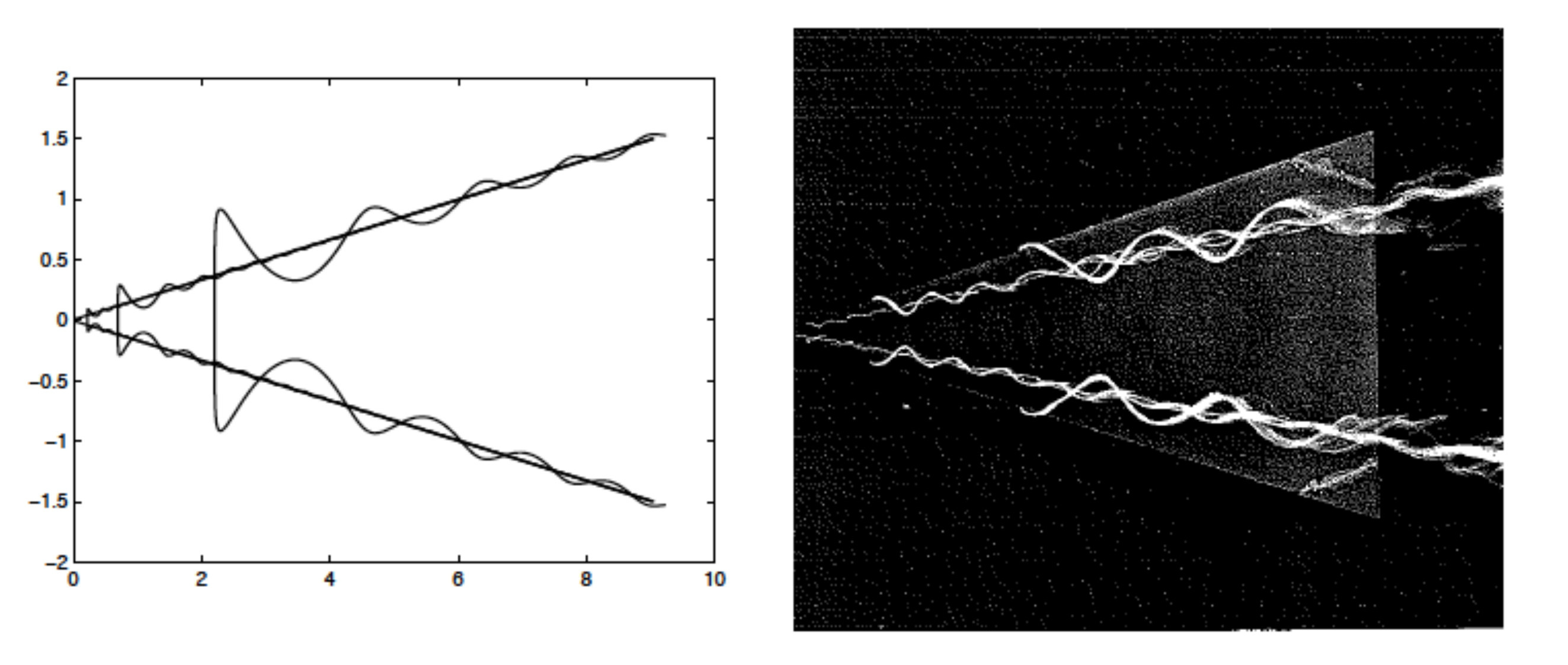} ,
  \caption{Comparison between a self-similar solution of the binormal flow and the experiment of a coloured fluid traversing a delta wing (from \cite{HGCV}).}
    \label{delta wing}
\end{figure} 
Next we are going to recall the results we have obtained in our previous papers \cite{BV1}, \cite{BV2} and \cite{BV3} about the stability of the self-similar solutions. These results will play a fundamental role in the proofs of the theorems we state in this article. 
 
We start noticing  that in the particular case of $\chi_a$ we have that for $t>0$ its filament function is 
\begin{equation}
\label{eq1.3}
\psi_a(x,t)=a\frac {e^{i\frac {x^2}{4t}}}{\sqrt {t}},
\end{equation}
which solves \eqref{Hasimoto} with $A(t)=\frac {|a|^2}{t}$,
\begin{equation}\label{NLSfilament}
i\psi_t+\psi_{xx}+\frac{\psi}{2}\left(|\psi|^2-\frac{a^2}{t}\right)=0.
\end{equation}
Observe that $|\psi_a|^2=a^2/t$ and therefore $\psi_a$ does not belong to $L^2(\mathbb R)$, but is just locally in $L^2$. A simple way of finding a natural function space such that $\psi_a$  belongs to it is to use the so-called  pseudo-conformal transformation 
\begin{equation}\label{calT}
\psi(t,x)=\mathcal {T}v(t,x)=\frac{e^{i\frac{x^2}{4t}}}{\sqrt{t}}\overline{v}\left(\frac 1t,\frac xt\right).
\end{equation}
If $\psi$ is a solution of \eqref{Hasimoto} then $\,v\,$ solves 
\begin{equation}\label{GP1}
iv_t+v_{xx}+\frac {1}{2t}\left(|v|^2-a^2\right)v=0.
\end{equation}
In particular for $\psi_a$  we obtain the constant solution $v_a=a$. This new equation has the associated energy
$$E(t)=\int |v_x|^2-\frac {1}{4t}\left(|v|^2-a^2\right))^2\,dx,$$
with
$\frac{d}{dt} E(t)=-\frac {1}{4t^2}\int\left(|v|^2-a^2\right))^2\,dx,$ and $E(t)=0$ for $v_a$.

We want to consider small perturbations   of $v_a$. Then, we write $v=u+a$ so that $u$ must be a solution of the following equation
\begin{equation}\label{NLS}
iu_t+u_{xx}+\frac{a+u}{2t}(|a+u|^2-a^2)=0.
\end{equation}
As a conclusion, understanding the large time behaviour of the solutions of \eqref{NLS} is equivalent  to understanding the behaviour of the perturbations of $\psi_a$ in \eqref{NLSfilament} at time $t=0$ which in turn is related to the behaviour of small  perturbations of the self-similar solution $\chi_a$ at the time  that the corner is created.
\medskip

In \cite{BV1} we start our study  of the scattering properties for equation \eqref{NLS}. 
In particular we obtain a first result about the existence of the wave operator.  For  $s\in\mathbb{N}^*$ we denote $H^s$ the usual Sobolev space in $\mathbb R$ of $L^2$ functions with  $s$-derivatives in $L^2$, $W^{s,1}$ the space of $L^1$ functions with $s$-derivatives in $L^1$, and  $\dot{H}^{s}$ the corresponding homogeneous space of functions with $s$-derivatives in $L^2$.  Finally $\dot{H}^{-2}$ denotes the set of tempered distributions $\phi$ such that 
\begin{equation}
\label{homog}
\int |\hat \phi(\xi)|^2\frac{d\xi}{|\xi|^4}< \infty.
\end{equation}
 Then, we show in \cite{BV1} that  for any $a$ small and any $f_+$ small in $\dot{H}^{-2}\cap H^{s}\cap W^{s,1}$, there exists a unique $u\in \mathcal{C}([1,\infty),H^s(\mathbb{R}))$ solution of \eqref{NLS} having $f_+$ as asymptotic state: for any $1\leq k\leq s$, 
\begin{equation*}
\sup_{1\leq t}\sqrt{t}\left\|u(t)-e^{i\frac {a^2}2\log t}e^{i(t-1)\partial_x^2}f_+\right\|_{L^2}+\sup_{1\leq t}t\left\|u(t)-e^{i\frac {a^2}2\log t}e^{i(t-1)\partial_x^2}f_+\right\|_{\dot H^k}\leq C(a,f_+).
\end{equation*}
Moreover if $f_+\in \dot{H}^{-2}\cap H^{3}\cap W^{3,1}$ and $x^2f_+\in L^2$ then we construct in \cite{BV1}  
{\it {some}} 
perturbations $\chi(t,x)$ of $\chi_a$, that are solutions of the binormal flow on $0\leq t\leq 1$, and that still have a ``corner" at time $0$. So in \cite{BV1} we proved that the development of a singularity in finite time for the self-similar solutions of \eqref{binormal} is not an isolated phenomena. Although assumption \eqref{homog} is very strong,  we  obtain the extra bonus of proving that  not just the perturbed solution remains close to the self-similar solution but also that the full Frenet frame is close to the starting one. In particular the binormal vectors also remain close and therefore from \eqref{binormal1} we obtain a much more precise information about the velocity of the perturbed filament.\par

 In \cite{BV2} we are able to avoid the assumption \eqref{homog} and the smallness hypothesis on $a$. For doing so we introduce   some function spaces that give special consideration to the low Fourier modes. More concretely we consider
\begin{equation}\label{Xtau}
\|f(x)\|_{X_{t_0}^\gamma}=\frac{1}{t_0^\frac 14}\|f\|_{L^2}+\frac{t_0^\gamma}{\sqrt{t_0}}\||\xi|^{2\gamma}\hat{f}(\xi)\|_{L^\infty(\xi^2\leq 1)}<+\infty
\end{equation}
and ${Y_{t_0}^\gamma}$ the space of
functions 
\begin{equation}\label{Ytau}
\|g(t,x)\|_{Y_{t_0}^\gamma}=\sup_{t\geq t_0}\,\left(\frac{1}{t_0^\frac 14}\|g(t)\|_{L^2}+\left(\frac{t_0}{t}\right)^{a^2}\frac{t_0^\gamma}{\sqrt{t_0}}\||\xi|^{2\gamma}\hat{g}(t,\xi)\|_{L^\infty(\xi^2\leq 1)}\right)<+\infty;
\end{equation}
for simplicity we shall drop in the notations the subindex $t_0$ when $t_0=1$. Let us notice that in view of Lemma 6.1 in \cite{BV3} the following results are valid in spaces $Y^\gamma$ where the power $a^2$ is replaced by any $\delta>0$. 
We have proved global existence and asymptotic completeness for initial data in $X^\gamma$. More precisely, for $s\in\mathbb N$ and for small initial data $\partial^k u(1,x) \in X^\gamma$, $0\leq k\leq s$, with $0<\gamma<\frac 14$, we proved that there exists a unique solution $u$ of \eqref{NLS}, with $\partial^k u\in Y^\gamma\cap L^4((1,\infty),L^\infty)$, and there exists $f_+\in H^s$ for which
\begin{equation*}
\sup_{1\leq t}t^{\frac14-\gamma^+}\left\|u(t)-e^{i\frac {a^2}2\log t}e^{i(t-1)\partial_x^2}f_+\right\|_{H^s}\leq C(a,u(1)),
\end{equation*}
and $\partial ^k f_+$ belongs to 
$X^{\gamma^+}.$ Moreover, we constructed wave operators in the Appendix of \cite{BV2} without smallness assumption on $a$:  for asymptotic states $f_+$ with $\partial^k f_+$ small in $X^\gamma$, there exists a unique solution $u$ of \eqref{NLS}, with $\partial^ku\in Y^{\gamma^+}\cap L^4((1,\infty),L^\infty)$ such that
\begin{equation*}
\sup_{1\leq t}t^{\frac 14-\gamma^+}\left\|u(t)-e^{i\frac {a^2}2\log t}e^{i(t-1)\partial_x^2}f_+\right\|_{H^s}\leq C(a,f_+).
\end{equation*}\\
As a consequence of the asymptotic completeness, at the level of the binormal flow \eqref{binormal} we have obtained in \cite{BV2} that in our functional setting {\it{all}} small perturbations at time $t=1$ of $\chi_a$ will end up generating a singularity in finite time at $t=0$. Nevertheless, we do not get too much geometric information about the trace $\chi(0,x)$ of $\chi(t,x)$  at $t=0$: for instance we do not obtain the behavior of $\chi(0,x)$ near $x=0$.\par
In \cite{BV3} we consider the solutions $u(t)$ constructed in \cite{BV2} via asymptotic completeness,  and look at the corresponding perturbations $\chi(t)$ of a self-similar solution $\chi_a$, started at time $t=1$. Adding the extra assumption   that the initial datum $u(1)$ belongs to an appropriate {\it{weighted space}}  we are able to get a precise asymptotic in space and in time of the tangent and normal vectors of $\chi(t)$. This allows  us to prove the stability of the self-similar structure of $\chi_a(t)$, as well as a complete description of the trace at time $t=0$ of $\chi(t)$. In particular we prove that the same corner as the one of $\chi_a(0)$ is created independently of the perturbation. \par
 Two main questions remain open after the paper \cite{BV3}. One is if it is possible to solve the binormal flow forward in time starting with a datum that has a corner at one point. In other words to prove that the initial value problem is  well posed for data that are regular except at one point where they have a corner. The second one is wether or not when going backward in time, and once the corner has been created, the solution can be continued for negative times. We answer positively to both questions in this paper. The main obstruction we have to bypass is the use that we make in \cite{BV3} of weighted spaces because, as we will see in the Appendix, they are spaces that  the scattering operator of the linearized equation associated to \eqref{NLS} does not leave invariant.

\medskip

Our main results are the following ones.
\begin{theorem}\label{ivp} (The initial value problem) Let $\chi_0$ be a smooth curve of class $\mathcal C^4$, except at $\chi_0(0)=0$ where a corner is located, i.e. that there exist $A^+$ and $A^-$ two distinct non-colinear unitary vectors in $\mathbb R^3$ such that
$$\chi_0'(0^+)=A^+,\quad \chi_0'(0^-)=A^-.$$
We set $a$ to be the parameter of the unique self similar solution of the binormal flow with the same corner as $\chi_0$ at time $0$. 

We suppose $\chi_0$ to be such that its curvature $c(x)$ for $x\neq 0$ satisfies $(1+|x|^4)c(x)\in L^2$ and $|x|^{2\gamma}c(x)\in L^\infty_{(|x|\leq 1)}$ small with respect to $a$  for some $0<\gamma<\frac 14$.

Then there exists 
$$\chi(t,x)\in\mathcal C([-1,1],Lip)\cap \mathcal C([-1,1]\backslash \{0\},\mathcal C^4),$$
regular solution of the binormal flow \eqref{binormal} for $t\in [-1,1]\backslash \{0\}$, having $\chi_0$ as limit at time $t=0$. Above $Lip$ denotes the set of Lipschitz functions.\par 
Moreover, the solution $\chi$ is unique in the subset of $\mathcal C([-1,1],Lip)\cap \mathcal C([-1,1]\backslash \{0\},\mathcal C^4)$ such that the associated filament functions at times $\pm1$ can be written as $(a+u(\pm1,x))e^{i\frac{x^2}{4}}$ with $u(\pm1)$ small in $X^\gamma\cap H^4$ with respect to $a$ for some $0<\gamma<\frac 14$.

This solution enjoys the following properties:\\
i) There exists a constant $C>0$ such that for $t\in[-1,1]$ we have the rate of convergence 
\begin{equation}\label{convbin}\sup_x |\chi(t,x)-\chi_0(x)|\leq C\sqrt{|t|}.
\end{equation}
ii) For all fixed $t_1, t_2\in[-1,1]\backslash \{0\}$ the following asymptotic properties hold  
\begin{equation}\label{convinfinity} \chi(t_1,x)-\chi(t_2,x)=\mathcal O\left(\frac{1}{x}\right),\qquad T(t_1,x)- T(t_2,x)=\mathcal O\left(\frac{1}{x}\right).\end{equation}
Moreover, there exists $T^\infty\in\mathbb S^2$ such that uniformly in $t\in[-1,1]$, for $x$ positive,
\begin{equation}\label{convinfinityT} T(t,x)-T^\infty=\mathcal O\left(\frac{1}{\sqrt{x}}\right).\end{equation}
iii) $\chi$ is a solution of the binormal flow for $t\in [-1,1]$ in the following weak sense
\begin{equation}\label{binweak}\int_{-1}^1\int \chi_t(t,x)\,\phi (t,x)\,dx\,dt =\int_{-1}^1\int \chi_x(t,x)\land\chi_{xx}(t,x)\,\phi (t,x)\,dx\,dt <\infty,
\end{equation}
for all test functions $\phi\in\mathcal{C}_0^\infty(\mathbb R^2)$. \\
iv) The tangent vector $T=\chi_x$ satisfies \eqref{schmap} for $t\in [-1,1]\backslash \{0\}$ and tends at $t=0$ to $T(0)=\chi'_0(x)=T_0$  for $x\neq 0$, with a rate of decay
\begin{equation}\label{convT}
\sup_{|x|>\epsilon>0}|T(t,x)-T_0(x)|\leq C_\epsilon |t|^{\frac 16^-}.
\end{equation}
v) The tangent vector $T$ is a solution of \eqref{schmap}  through $t=0$ in the following weak sense
\begin{equation}\label{Tweak}\int_{-1}^1\int T(t,x)\,\phi_t (t,x)\,dx\,dt =\int_{-1}^1\int T(t,x)\land T_{x}(t,x)\,\phi_x (t,x)\,dx\,dt <\infty,
\end{equation}
for all test functions $\phi\in\mathcal{C}_0^\infty(\mathbb R^2)$. 
\end{theorem}

\begin{theorem}\label{cont}(Continuation of solutions through the singularity time) Let $\chi(1)$ be a small perturbation of a self-similar solution $\chi_a$ at time $t=1$ in the sense that the filament function \eqref{eq1.2} of $\chi(1)$ is $(a+u(1,x))e^{i\frac{x^2}{4}}$, with $u(1)$ small in $X^\gamma\cap H^4$ with respect to $a$ for some $0<\gamma<\frac 14$. Then, we can construct a regular solution $\chi$ for the binormal flow \eqref{binormal} on $ t\in [-1,1]\backslash \{0\}$, having at time $t=0$ a limit $\chi_0$ and enjoying the properties i)-v) of Theorem \ref{ivp}. 
Moreover, the corner of the self-similar solution is recovered: $\partial_s\chi(0,0^\pm)=\partial_s\chi_a(0,0^\pm)$.\par
This solution $\chi$ is unique in the subset of $\mathcal C([-1,1],Lip)\cap \mathcal C([-1,1]\backslash \{0\},\mathcal C^4)$ such that the associated filament functions at times $\pm1$ can be written as $(a+u(\pm1,x))e^{i\frac{x^2}{4}}$ with $u(\pm1)$ small in $X^\gamma\cap H^4$ with respect to $a$ for some $0<\gamma<\frac 14$. \end{theorem}

Let us briefly explain the proof of Theorem \ref{ivp}. We recall the notation $B^\pm$ for the complex vector appearing in the asymptotics of the normals vectors of the unique self similar solution of the binormal flow with the same corner as $\chi_0$ at time $0$ (see Theorem \ref{sssols}). We denote $T_0=\chi_0'$. We define for $x>0$  a complex-valued function $g$ and a $\mathbb C^3$-valued function $\tilde N_0$ orthonormal to $T_0$ by solving  the system 
\begin{equation}\label{systinitial}
\left\{\begin{array}{c}T_{0x}(x)=\Re (g(x)\tilde N_0(x)),\\ \tilde N_{0x}(x)=-\overline{g}(x)T_0(x),
\end{array}\right.\end{equation}
with initial data $(A^+,B^+)$. We define $g(x)$ and $\tilde N_0$ similarly for $x<0$ imposing $(A^-,B^-)$ as initial data in \eqref{systinitial}. In particular we have the following link with the curvature of $\chi_0$: $|g(x)|=c(x)$. Therefore $(1+|x|^4)g(x)\in L^2$ and $|x|^{2\gamma}g(x)\in L^\infty_{(x^2\leq 1)}$ are small with respect to $a$. Next we define 
$$f_+=\mathcal F^{-1}\left(g(2\cdot)e^{ia^2\log |2\cdot|}\right).$$ 
In particular $f_+$ and its first four derivatives are small in 
$X_1^\gamma$ with respect to $a$. This allows us to obtain $u(t)$ the solution of \eqref{NLS} with asymptotic state $f_+$, given by the construction of wave operators in \cite{BV2}. We set $\chi$ to be the corresponding binormal flow solution (for the construction see for instance the Appendix of \cite{BV1}). It was also showed in \cite{BV2} that $u(1)\in X^{\gamma^+}$. We shall prove that we can carry on $u(t)$ the computations done in \cite{BV3}, so we can define for $\chi$ a trace at time zero $T(0,x)$. We recall that in \cite{BV3} we were working with solutions $u(t)$ generated by initial data at finite time $t=1$.  We used that the weight condition $x^2u(t)\in L^2$ holds, something that is satisfied if we assume it at initial time $t=1$. But now we have to go backwards in time from the asymptotic state $f_+$ to the solution $u(t)$ and as we already said there seems to be a serious obstruction for showing that in this case $u(t)$ is in weighted spaces; we give the details in the Appendix. In \S \ref{sectcond}-\ref{sectT0} we shall perform on $u(t)$ the computations done in \cite{BV3}, in such a way that we can avoid the assumptions on weights. Finally we shall prove in \S \ref{sectT0bis} that the trace $T(0,x)$ coincides with $T_0(x)$, which will give us the solution of Theorem \ref{ivp} for $t\geq 0$. Our uniqueness result rely on the existence and uniqueness of the solution of the associated
Frenet system and NLS equations. For negative times we shall do the same, starting from
$$\mathcal F^{-1}\left(\bar g(-2\cdot)e^{ia^2\log |2\cdot|}\right).$$ 
We shall find similarly a solution $\chi^*$ of the binormal flow with initial data  $\chi^*(0,x)=\chi(0,-x)$. Then for negative times we shall set $\chi(-t,x)=\chi^*(-t,x)$ to obtain the solution in Theorem \ref{ivp} on $[-1,1]$.

Concerning Theorem \ref{cont} we recall that its part concerning positive times $t\geq 0$ was the main result in \cite{BV3}, under the assumption that weighted conditions are satisfied for $u(1)$. As we have said, in \S \ref{sectcond}-\ref{sectT0} we shall remove these conditions. For extending $\chi$ to negative times, we shall proceed as explained above for Theorem \ref{ivp}.

\medskip

The paper is organized as follows. In the following section we shall recall the results in \cite{GRV} on self-similar solutions of the binormal and describe the continuation through time $0$.  In section \S \ref{sectivp} we shall give the proof of Theorem \ref{ivp} while Theorem \ref{cont} will be treated in section \S \ref{sectcont}. The Appendix will contain results and remarks on the equation \eqref{NLS} in weighted spaces, via the so-called $J$-operators, $J(t)=x+it\partial_x$.

\medskip 
{\bf{Acknowledgements:}} First author was partially supported by the French ANR projects  R.A.S. ANR-08-JCJC-0124-01 and SchEq ANR-12-JS-0005-01. The second author was partially supported by the grants UFI 11/52, MTM 2007-62186 of MEC (Spain) and FEDER.\\

\section{Self-similar solutions of the binormal flow}
In this section we review the known results on self-similar solutions of the binormal flow, i.e.
$$\chi(t,x)=\sqrt{t}\,G\left(\frac x{\sqrt{t}}\right),\qquad t\geq 0,$$
and focus on the issue of their possible extension  for negative times.
These solutions have been investigated first by physicists in the 80's (\cite{LD}, {\cite{LRT}, \cite{Bu}). Using the Frenet equations they observed that, modulo rotations,  self-similar solutions form a one parameter family $\chi_a$ of curves with $a\in\mathbb R^{+*}$ such that the curvature and the torsion of $\chi_a(t)$ at $x$ are $\frac{a}{\sqrt{t}}$ and $\frac{x}{2t}$ respectively. The mathematical rigorous description was given in \cite{GRV}.  Using the expressions of the derivative in time of the tangent and normal vectors of a solution of the binormal flow, one gets that $(T_a,n_a,b_a)(t,0)$ is constant in time. Therefore by integrating the binormal flow at $t=0$ it follows that 
\begin{equation}
\label{fundamental}\chi_a(t,0)=2a\sqrt{t}\,b_a(t,0). 
\end{equation}
In particular, the curve profile $G_a(x)$ satisfies $G_a(0)=2a\,b_a(t,0)$, so the only degree of freedom in constructing a self-similar solution is in the choice of the Frenet frame at $x=0$. Theorem 1 of \cite{GRV} states that given $a\in\mathbb R^{+*}$ there exists a unique frame $(T_a,n_a,b_a)(t,x)$ solution of the Frenet system of equations
\begin{equation}\label{frenet}
\left\{\begin{array}{c}T_{x}= cn,\\ 
 n_{x}=-cT+\tau b\\
 b_x=-\tau b,
\end{array}\right.\end{equation}
with  the curvature and the torsion $(c_a,\tau_a)(t,x)= (\frac{a}{\sqrt{t}},\frac{x}{2t})$ and taking the canonical basis of $\mathbb R^3$ as the initial data at $(t,0)$.  As a consequence there is a unique self-similar solution $\chi_a$ of the binormal flow such that its Frenet frame at $x=0$ is the canonical basis of $\mathbb R^3$. This solution is written as 
\begin{equation}
\label {chi}
\chi_a(t,x)=\sqrt{t}\,G_a\left(\frac x{\sqrt{t}}\right),\qquad t\geq 0,
\end{equation}
with the profile $G_a(\frac x{\sqrt{t}})=2ab_a(0,0)+\int_0^\frac x{\sqrt{t}} T(t,s)\,dx$. Then \eqref{fundamental} becomes
\begin{equation}
\label{fundamental1}\chi_a(t,0)=2a\sqrt{t}(0,0,1), \qquad t\geq0. 
\end{equation}

Moreover, in \cite{GRV} a precise description of the profile $G_a(s)$ is given for large $|s|$ (here $s$ stands for the self-similar variable: $s=\frac{x}{\sqrt{t}}$). This aymptotic plays a crucial role in the proof of Theorem 1.1 in \cite{BV3} as well as in the proofs of Theorem \ref{ivp} and Theorem \ref{cont} of this paper.  More concretely in \cite{GRV} the following result is proved.\begin{theorem}\label{thGRV}(\cite{GRV}) Given $a>0$ then the family of curves
$$\chi_a(t,x)=\sqrt{t}\,G_a\left(\frac x{\sqrt{t}}\right),$$
with $G_a$ given in \eqref{chi} (i.e. the Frenet frame $(T_a,n_a,b_a)$ at $x=0$ is the canonical orthonormal basis of $\mathbb R^3$ ) is a solution of the binormal flow which is real analytic for $t>0$. Moreover, there exist $A^\pm_a$ and $B^\pm_a$ such that\\
(i)$$\quad |\chi_a(t,x)-A^+ \,x\, {\bf{1}}_{[0,\infty[}(x)-A^- \,x\, {\bf{1}}_{]-\infty,0]}(x)|\leq a\sqrt{t},$$
(ii) The following asymptotics hold, for $s\rightarrow \pm\infty$:
$$G_a(s)=A^\pm_a\left(s+\frac {2a^2}{s}\right)-\frac{4a}{s^2}\,n_a+\mathcal O\left(\frac{1}{s^3}\right),$$
$$T_a(s)=A^\pm_a-\frac{2a}{s}\,b_a+\mathcal O\left(\frac{1}{s^2}\right),$$
$$(n_a-ib_a)(s)=B^\pm_a\,e^{i\frac{s^2}{4t}}\,e^{ia^2\log |s|}+\mathcal O\left(\frac{1}{s}\right),$$
(iii) The real vectors $A^\pm_a=(A^\pm_{a,1},A^\pm_{a,2},A^\pm_{a,3})$ are unitary and
$$A^+_{a,1}=A^-_{a,1}=e^{-\pi\frac{a^2}{2}}, \quad A^+_{a,2}=-A^-_{a,2}, \quad A^+_{a,3}=-A^-_{a,3},\quad A_a^\pm\perp B_a^\pm,$$ 
(iv)  The complex vectors $B^\pm_a=(B^\pm_{a,1},B^\pm_{a,2},B^\pm_{a,3})$ verify $|\Re B^\pm_a|=|\Im B^\pm_a|=1$ and
\begin{equation}
\label{Bs}B^+_{a,1}=-B^-_{a,1}, \quad B^+_{a,2}=B^-_{a,2}, \quad B^+_{a,3}=B^-_{a,3},
\end{equation}
(v) The angle of the corner of $\chi_a(0)$ is determined by
$$\sin\frac{(\widehat{A^+,-A^-})} 2=A^\pm_{a,1}=e^{-\pi\frac{a^2}{2}}.$$
\end{theorem}

Keeping in mind that the binormal flow is invariant under rotations, Theorem \ref{sssols} is a reformulation of part of this theorem.

As mentioned in the Introduction,  there is a unique form to continue the solution $\chi$ in Theorem \ref{sssols}  for negative times $t<0$ in a self-similar way. This is done taking   $\chi_a(t,x)=\chi_a^*(-t,-x)$, where $\chi_a^*$  is a solution of the binormal flow for positive times with initial data $\chi_a(0,-x)$. Theorem \ref{sssols} ensures us that $\chi_a^*$ is unique. As a consequence, the unique way to extend $\chi_a$ for negative times is to perform  a rotation of $\chi_a$ around the axis given by  $A_a^+-A_a^-$ and with angle $\pi$. This rotation, that we shall call $\rho_a$, can be also  seen as a composition of a reflection with respect to the plane generated by $A_a^+$ and $A_a^-$, and a change of the sense of parametrization. Moreover, the trajectory of the origin for negative times is given by
\begin{equation}
\label{fundamental2}\chi_a(t,0)=2a\sqrt{|t|}\,\rho_a(0,0,1)\qquad t<0.
\end{equation}
That is to say, the trajectory $\chi_a(t,0)$ is given by two lines that join together at $t=0$ with an angle that is determined by $A^+_{a,2}$, see Theorem 1 in \cite{GRV}.

Hereafter and for simplicity we shall drop the subindex $a$ that we have used so far to parametrize the family of self-similar solutions.

We denote $\Pi$ the plane generated by $A^+$ and $A^-$ and by $\Pi^o$ the orthogonal plane generated by $A^+-A^-$ and $A^+\land A^-$. We also introduce $\Pi^\pm$ the plane generated by $\Re B^\pm$ and $\Im B^\pm$. Since $B^\pm\perp A^\pm$ we have that $\Pi^\pm$ is the orthogonal plane to $A^\pm$.
 We shall need the following proposition in the proof of Theorem \ref{ivp} and Theorem \ref{cont}.
 
 \begin{prop}\label{contssim}
 The self similar solution $\chi^*$ of the binormal flow with initial data 
 $$\chi^*(0,x)=\chi(0,-x)=\left\{\begin{array}{cc}-A^-x,\,x\geq 0,\\-A^+x,\, x\leq 0.\end{array}\right.$$
 has the following properties
 $$A^{\pm*}=\rho(A^\pm)=-A^\mp=R^\mp(-A^\mp),\quad B^{\pm*}=\rho(B^\pm)=R^\mp\overline{B^\mp},$$
 where $R^\mp$ is a rotation of angle $2\theta$ in the plane $\Pi^\mp$ and $\theta$ is the angle between $\Re B^\mp$ and the plane $\Pi$. 
 \end{prop}
 
 \begin{figure}
    \includegraphics[height=8cm]{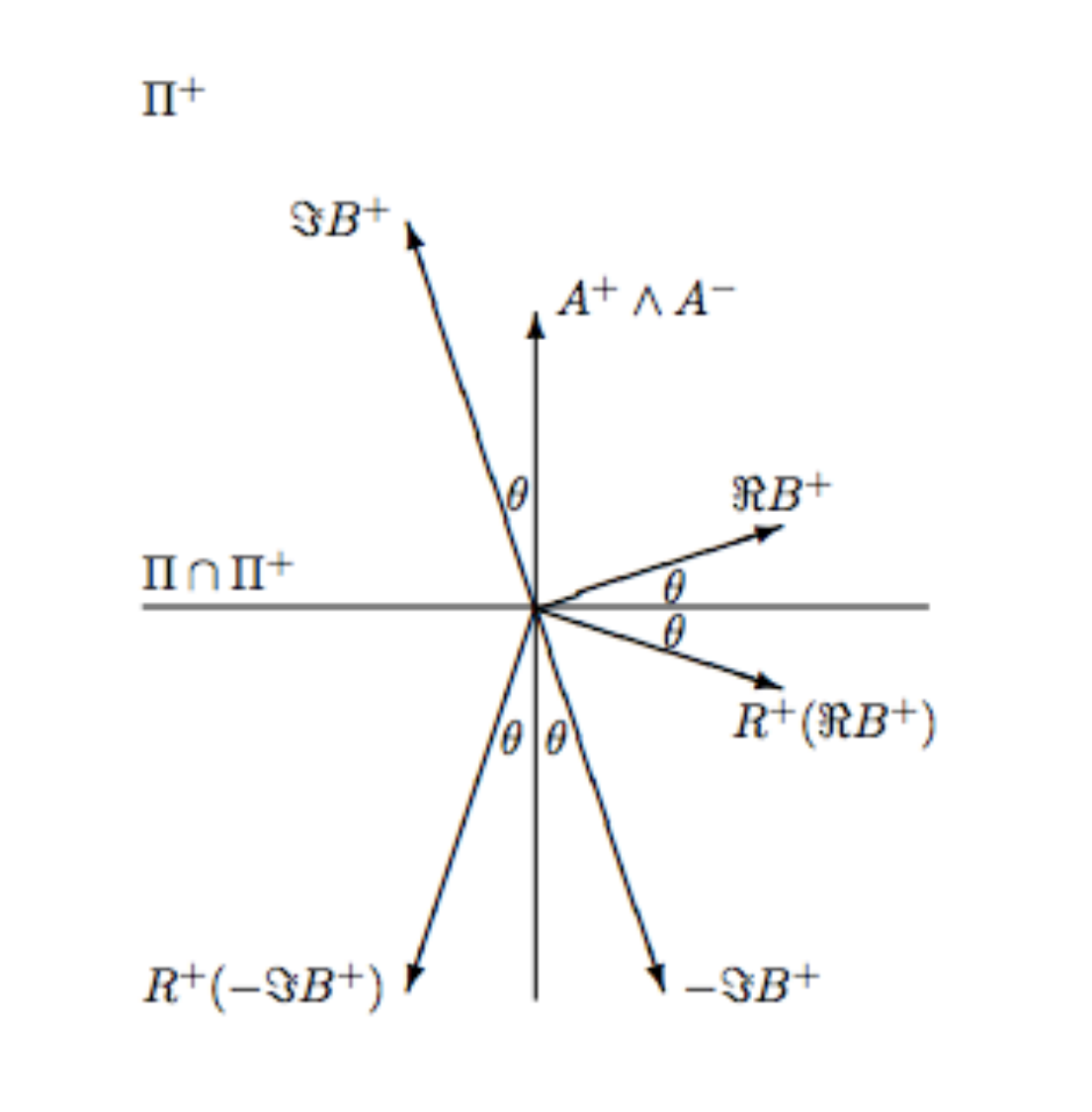} ,
  \caption{The plane $\Pi^+$.}
    \label{Bplane}
\end{figure}

\begin{proof}
We have already seen that $\chi^*(t,x)=\rho\chi(t,x)$ so that $(T^*,n^*,b^*)(t,x)=\rho\,(T,n,b)(t,x)$ and $(c^*,\tau^*)(t,x)=(c,\tau)(t,x)$. Then, it is easy to see that $T^*(t,x)$ goes to $-A^-$ as $x$ goes to $+\infty$, and to $-A^+$ as $x$ goes to $-\infty$. Since $\rho(A^\pm)=-A^\mp$, we obtain that $A^\pm\,^*=-A^\mp$. We are left with seeing what is $B^{+*}$ in terms of $B^\pm$. 
 
 By definition, since the torsion of $\chi(t,s)$ is $\frac {s}{2t}$,
$$B^{\pm}=\underset{x\rightarrow \pm\infty}{\lim}(n+ib)(t,x)e^{i\frac{x^2}{4t}}e^{-ia^2\log\sqrt{|t|}\,+ ia^2 \log|x|}=\underset{x\rightarrow \pm\infty}{\lim}N(t,x)\,e^{-ia^2\log\sqrt{|t|}\,+ ia^2  \log|x|}.$$
Since $(n^*,b^*)(t,x)=\rho\,(n,b)(t,x)$ and $\tau^*(t,x)=\tau(t,x)$,
$$N^*(t,x)=(n^*+ib^*)(t,x)e^{i\int_0^x\tau^*(t,s)ds}=(\rho n+i\rho b)(t,x)e^{i\frac{x^2}{4t}}=\rho\,N(t,x).$$
In particular
$$B^{\pm*}=\rho B^\pm.$$
From \eqref{Bs} we conclude that  $B^+$ is a reflection of $B^-$ with respect to the  plane $x_1=0$, which is precisely the plane $\Pi^o$. The rotation $\rho$ can be also seen as a composition of a reflection with respect to the plane $\Pi$ with a reflection with respect to the plane $\Pi^o$. Therefore  $\rho B^-$ is a reflection of $B^+$ with respect to the plane $\Pi$. In view of the definition of $B^+$ we obtain $\Re B^+\perp \Im B^+$. It follows then that the reflection of $\Re B^+$ with respect to the plane $\Pi$ is $R^+ \Re B^+$, and the one of $\Im B^+$ is $R^+(-\Im B^+)$. As a conclusion $B^{-*}=R^+\overline{B^+}$ (see Figure \ref{Bplane}). Moreover, since $A^+$ is orthogonal to $\Pi^+$, it follows that $A^{-*}=-A^+=R^+(-A^+)$. Again, since $B^+$ is a reflection of $B^-$ with respect to $\Pi_0$, the angle between $B^-$ and the plane generated by $A^+$ and $A^-$ is also $\theta$ and similarly we get $B^{+*}=R^- \overline{B^-}$ and $A^{+*}=-A^-=R^-(-A^-)$. 
 \end{proof}

\section{Proof of Theorem \ref{ivp}}\label{sectivp}
As announced in the Introduction, we construct a function $f_+$ from the curve $\chi_0$ in Theorem \ref{ivp} as follows.  We recall the notation $B^\pm$ for the complex vector appearing in the asymptotics of the normals vectors of the unique self similar solution of the binormal flow with the same corner as $\chi_0$ at time $0$ (see Theorem \ref{sssols}). We denote $T_0=\chi_0'$. We define for $x>0$  a complex-valued function $g$ and a $\mathbb C^3$-valued function $\tilde N_0$ orthonormal to $T_0$ by solving  the system 
\begin{equation}\label{systinitial}
\left\{\begin{array}{c}T_{0x}(x)=\Re (g(x)\tilde N_0(x)),\\ \tilde N_{0x}(x)=-\overline{g}(x)T_0(x),
\end{array}\right.\end{equation}
with initial data $(A^+,B^+)$. We define $g(x)$ and $\tilde N_0$ similarly for $x<0$ imposing $(A^-,B^-)$ as initial data in \eqref{systinitial}. In particular we have the following link with the curvature of $\chi_0$: $|g(x)|=c(x)$. Therefore $(1+|x|^4)g\in L^2$ and $|x|^{2\gamma}g(x)\in L^\infty_{(x^2\leq 1)}$ are small with respect to $a$. Next we define 
$$f_+=\mathcal F^{-1}\left(g(2\cdot)e^{ia^2\log |2\cdot|}\right).$$ 
In particular $f_+$ and its first $4$ derivatives are small in $X^\gamma$ with respect to $a$. We let $u(t,x)$ be the solution of \eqref{NLS} with asymptotic state $f_+$, given by the construction of wave operators in \cite{BV2}. It was also shown in \cite{BV2}  that $u(1,x)$ and its first $4$ derivatives belong to $X^{\gamma^+}$. The following bounds hold for all $ 0\leq k\leq 4$ and $t_1\leq t_2<\infty$
\begin{equation}\label{unifbounds}
\exists C(a),\tilde C(a),\, \|\partial_x^ku(1/t)\|_{L^\infty[t_1,t_2],H^1}\leq C(a)\sum_{j=0}^4\|\partial_x^ju(1)\|_{X^{\gamma^+}}\leq \tilde C(a)\sum_{j=0}^4\|\partial_x^jf_+\|_{X^{\gamma}}.
\end{equation}
Next we define $\psi$ by the pseudo-conformal transformation \eqref{calT},
 $$\psi(t,x)=\mathcal T (a+u)(t,x)=\frac{e^{i\frac{x^2}{4t}}}{\sqrt{t}}\overline{a+u}\left(\frac 1t,\frac xt\right).$$
Finally we construct $\chi(t,x)$ to be the corresponding solution of the binormal flow for $t\geq0$, i.e. with curvature $c(t,x)=|\psi(t,x)|$ and torsion $\tau(t,x)=\partial_x \arg \psi(t,x)$, having as initial data at time $t_0>0$ the location $\chi(t_0,0)=(0,0,0)$ and as Frenet frame $(T,n,b)(t_0,0)$ the canonical orthonormal basis of $\mathbb R^3$\footnote{We denote $(e_1, e_2, e_3)$ the orthonormal basis of $\mathbb R^3$ and $(T,N)(t_0,0)=(e_1,e_2+ie_3)$. First, we construct $(T,N)(t,x)$ by imposing the evolutions laws 
$$T_x=\Re(\overline \psi N), N_x=-\psi T, T_t=\Im\overline{\psi_x}N,\quad N_t=-i\psi_x T+i(|\psi|^2-A(t))\, N.$$ 
Then $\chi(t,x)$ defined as $
\chi(t_0,x)=(0,0,0)+\int_t^{t_0}(T\wedge T_{xx})(\tau,0)d\tau+\int_0^xT(t,s)ds,
$
is a solution of the binormal flow. For details one can see for instance the Appendix of \cite{BV1} where the same type of costruction is done using the Frenet frame instead of the $(T,N)$ frame, the link between the two constructions being that the two frames are related by the normal rotation $N(t,x)=(n+ib)(t,x)e^{i\int^x_0 \tau(t,s)ds}$.} The curve $\chi(t)$ has curvature close to $\frac a{\sqrt{t}}$, and since it  satisfies the binormal flow it follows that it has a trace at time $t=0$. In particular $\chi(0,0)$ is a point in $\mathbb R^3$. We translate $\chi$ in space such that  $\chi(0,0)=\chi_0(0)$. Let $(T,n,b)(t,x)$ be its (very oscillating) Frenet frame for $t>0$ and consider the following complex normal vectors
$$N(t,x)=(n+ib)(t,x)e^{i\int^x_0 \tau(t,s)ds}\,\,,\,\tilde N=Ne^{i\Phi} \mbox{ with } \Phi(t,x)=-a^2\log\sqrt{t}+a^2 \log|x|.$$
We shall prove in the next two subsections that for $x\neq 0$ the tangent vector $T(t,x)$ has a limit as $t$ goes to $0$, and eventually in \S \ref{sectT0bis} that modulo a rotation this limit is precisely $T_0(x)$, so modulo a rotation $\chi(0,x)=\chi_0(x)$. Then we shall show the uniqueness of $\chi$. Finally, in \S \ref{sectext} we shall extend $\chi(t,x)$ for negative times and end the proof of Theorem \ref{ivp}.

\subsection{Asymptotic behaviour in time and space for the tangent vector}\label{sectcond}
We start first with an asymptotic analysis of tangent and normal vectors, keeping track of both time and space variables. 

\begin{prop}\label{propas}There exist $C>0$, $T^{\pm\infty}\in\mathbb S^2$ and $N^{\pm\infty}\in\mathbb C^3$ such that for all times $0<t\leq 1$, and $x\neq 0$, the following estimates hold, with the choice between $\pm$ given by the sign of $x$:
$$|T(t,x)-T^{\pm\infty}|\leq C\|\partial_xu(1)\|_{X^{\gamma^+}}\,\frac 1{\sqrt{|x|}}+C_1\,\frac{\sqrt t}{|x|},$$
$$\left|\tilde N(t,x)-N^{\pm\infty}\right|\leq C\|\partial_xu(1)\|_{X^{\gamma^+}}\,\frac 1{\sqrt{|x|}}+C_2\left(\frac{\sqrt t}{|x|}+\frac{t}{x^2}+\sqrt t\right).$$

Moreover,
\begin{equation}\label{Tasy}
T(t,x)-T^{\pm\infty}+\Im\int_x^{\pm\infty} h(t,s)\,\tilde N (t,s)ds=c_0(t,x),
\end{equation}
\begin{equation}\label{Nasy}
\tilde N(t,x)-N^{\pm\infty}-i\int_x^{\pm\infty} \overline{h(t,s)}\,T(t,s)\,ds=d_0(t,x),
\end{equation}
with
\begin{equation}\label{c0d0}
|c_0(t,x)|\leq C_1\frac{\sqrt t}{|x|},\quad |d_0(t,x)|\leq C_2\left(\frac{\sqrt t}{|x|}+\frac{t}{x^2}+\sqrt t\right),
\end{equation}
and the notations
$$h(t,s)= e^{-i\frac{s^2}{4t}}\frac{2}{s\sqrt t}\,(u_s)\left(\frac 1t,\frac st\right)e^{-i\Phi(t,s)},$$
$$C_0=\|u(1)\|_{X^{\gamma^+}}+\|\partial_xu(1)\|_{X^{\gamma^+}},\quad C_1=C(a+C_0), \quad C_2=C(a+a^4+(1+a^2)C_0+C_0^2).$$

\end{prop}

\begin{proof}
This result was proved in \cite{BV3} (see formulas (31) and (32)), provided that $u(1)$ belongs to some weighted space, which is not the case in the present paper. These weighted spaces were used in the proofs formulas (31) and (32) in \cite{BV3} only  for showing that the limit at infinity of the normal vector $\tilde N(t,x)$ is independent of time (Lemma 3.4 in \cite{BV3}), and more precisely in showing that
\begin{equation}\label{limit}\int_t^1\left|u\left(\frac 1{t'},\frac x{t'}\right)\right|\frac{dt'}{t'}\overset{x\rightarrow +\infty}{\longrightarrow}0.\end{equation}

For getting \eqref{limit} without using weights conditions, and hence obtaining the Proposition, we proceed in here as follows. We have
$$\int_t^1\left|u\left(\frac 1{t'},\frac x{t'}\right)\right|\frac{dt'}{t'}\leq \log t\sup_{t'\in[t,1]}\left|u\left(\frac 1{t'},\frac x{t'}\right)\right|,$$
so it is enough to prove that \footnote{It is easy to show that $$\sup_{t'\in[t,1]}\left|u\left(\frac 1{t'},x\right)\right|$$
tends to zero as $x$ goes to infinity by using the fact that $u\left(\frac 1{t'},x\right)\in W^{1,p}([t,1]\times[0,\infty))\subset L^\infty(\mathbb R^2)$ for $p>2$, then use the approximation of $W^{1,p}$ by $\mathcal C^\infty(\mathbb R^2)$ functions. The issue is that $u\left(\frac 1{t'},\frac x{t'}\right)$ is not in $W^{1,p}([t,1]\times[0,\infty)$ because when we compute its $\partial_{t'}$ partial derivative we get a factor $x$ and therefore weighted spaces are needed this way.}
$$\sup_{t'\in[t,1]}\left|u\left(\frac 1{t'},\frac x{t'}\right)\right|\overset{x\rightarrow +\infty}{\longrightarrow}0$$
Suppose it does not. Then 
$$\exists \epsilon>0,\, \exists x_n\rightarrow +\infty,\quad \sup_{t'\in[t,1]}\left|u\left(\frac 1{t'},\frac {x_n}{t'}\right)\right|>\epsilon.$$
In particular 
$$\exists t_n\in [t,1], \quad \left|u\left(\frac 1{t_n},\frac {x_n}{t_n}\right)\right|>\frac{99}{100}\epsilon.$$
Since $u\left(\frac 1{t_n}\right)$ is continuous, we obtain
$$\left\|u\left(\frac 1{t_n}\right)\right\|_{L^\infty\left(\frac{x_n}{t_n},\infty\right)}>\frac\epsilon 2.$$
Moreover, $t_n<1$ so
$$\left\|u\left(\frac 1{t_n}\right)\right\|_{L^\infty\left(x_n,\infty\right)}>\frac\epsilon 2.$$
Now $t_n\in[t,1]$ so there is a subsequence (that we recall $t_n$) and a number $t_0\in[t,1]$ such that $t_n\rightarrow t_0.$ We use $u\in\mathcal C\left(\left[1,\frac 1t\right],H^1\right)$ to get
$$\left\|u\left(\frac 1{t_n}\right)-u\left(\frac 1{t_0}\right)\right\|_{H^1}\overset{n\rightarrow +\infty}{\longrightarrow} 0,$$
so
$$\left\|u\left(\frac 1{t_n}\right)-u\left(\frac 1{t_0}\right)\right\|_{L^\infty}\overset{n\rightarrow +\infty}{\longrightarrow} 0.$$
There exists $N_\epsilon$ such that for all $n\geq N_\epsilon$ 
$$\left\|u\left(\frac 1{t_n}\right)-u\left(\frac 1{t_0}\right)\right\|_{L^\infty}\leq \frac \epsilon 4.$$
Since
$$\left\|u\left(\frac 1{t_n}\right)\right\|_{L^\infty(x_n,\infty)}> \frac \epsilon 2,$$
we obtain for all $n\geq N_\epsilon$
$$\left\|u\left(\frac 1{t_0}\right)\right\|_{L^\infty(x_n,\infty)}> \frac \epsilon 4.$$
Since $x_n$ tends to infinity, this is in contradiction with the fact that $u\left(\frac 1{t_0}\right)$ belongs to $H^1$. 

In conclusion, \eqref{limit} can be proved without weight conditions and the Proposition follows. 
\end{proof}

\subsection{The existence of the tangent vector at $t=0$}\label{sectT0}
In order to obtain the existence of a trace for the tangent vector at $t=0$ we would like to proceed in a way similar to the one in \cite{BV3} but avoiding the assumption that $u(t)$ is in weighted spaces. Hence we will re-express in formula \eqref{Tasy} the vector $\tilde N$ appearing in the integral by using \eqref{Nasy} to obtain an integral equation on $T$. Our plan is then to solve this equation by iteration as done in \cite{BV3}. 

We shall perform the analysis for $x>0$; the case $x<0$ goes the same.

Note that in \cite{BV3} we were able to iterate this process that generates multiple integrals because we proved 
 (Lemma 4.1 in \cite{BV3}) that for $t$ small with respect to $x$
$$\int_x^\infty|h(t,s)|ds\leq C_3 +C_4\,\frac{t^\frac 14}{x},$$
with
$$C_3=CC_0,\quad C_4=C(a)(C_0+\|xu(1)\|_{L^2}),$$
so $C_3$ and $C_4$ are small if $u(1)$ and its derivative are small enough in $X^{\gamma^+}$.
However, the proof of this key estimate relied on the fact that $u(t)$ belong to weighted spaces.  In order to avoid the use of weighted space we shall take advantage of the particular form of $h(t,s)$ that involves a derivative in space of $u$. We start by proving a lemma which will be frequently used in this subsection. \\

We recall that
$$h(t,s)= e^{-i\frac{s^2}{4t}}\frac{2}{s\sqrt t}\,(u_s)\left(\frac 1t,\frac st\right)e^{-i\Phi(t,s)} \mbox{ with } \Phi(t,s)=-a^2\log\sqrt{t}+a^2\log|s|.$$

\begin{lemma}\label{restest}
There exists a constant $C>0$ such that for all $n\in \mathbb N^*$, $0<t\leq 1$ and $0<x$ the following estimate holds
\begin{equation}\label{nintT}
\left|\int_x^\infty h_{i_1}(t,s_1)\int_{s_1}^\infty h_{i_2}(t,s_2) ...\int_{s_{n-1}}^\infty h_{i_n}(t,s_n)\,f(t,s_n)\,ds_n...ds_1\right|
\end{equation}
$$\leq C^n\|u(1/t)\|_{H^1}^n \left(1+\frac t{x^2}\right)^{n-1}\left( \left(1+\frac t{x^2}\right)\|f(t)\|_{L^\infty(x,\infty)}+\frac tx\,\|\partial_sf(t)\|_{L^2(\min\{x,1\},1)}\right),$$
where $h_{i_j}\in\{h,\overline h\}$ for $1\leq j\leq n$.
\end{lemma}

\begin{proof}

It will follow from the proof below that we can suppose without loss of generality that $h_{i_1}=...=h_{i_n}=h$. 

We shall prove the lemma by recursion on $n$. 
A trivial estimate can be obtained using  Cauchy-Schwarz inequality
\begin{equation}
\label{x>1}\left|\int_x^\infty h(t,s_1) \int_{s_1}^\infty h(t,s_2)...\int_{s_{n-1}}^\infty h(t,s_n)\,f(t,s_n)\,ds_n...ds_1\right|\leq  \frac{\|u_s(1/t)\|_{L^2}^n}{\sqrt{x}^n}\|f(t)\|_{L^\infty(x,\infty)}.
\end{equation}
So for $x\geq 1$ the lemma follows immediately. Eventually in the next subsections we will let $t$ tend to $0$, and $n$ to infinity, so such an upperbound is not satisfactory when $x<1$.

For $x\leq 1$ and $n=1$ the lemma was proved in \cite{BV3}, formula (38). For sake of completeness we recall its short proof here. We split the integral from $x$ to $1$ and from $1$ to $\infty$, and we perform an integration by parts on $[x,1]$,
$$\int_x^\infty h(t,s)\,f(t,s)ds=\int_1^{\infty} h(t,s)\,f(t,s) ds+2\sqrt{t}e^{-i\frac{1}{4t}-i\Phi(t,1)}\,u\left(\frac 1t,\frac 1t\right)\,f(t,1) $$
$$-\frac{2\sqrt{t}}{x}e^{-i\frac{x^2}{4t}-i\Phi(t,x)}\,u\left(\frac 1t,\frac xt\right)\,f(t,x)-\int_x^1u\left(\frac 1t,\frac st\right)\left(\frac{2\sqrt{t}}{s}e^{-i\frac{s^2}{4t}-i\Phi(t,s)}\,f(t,s)\right)_s.$$
We use Cauchy-Schwarz and the fact that $u$ belongs to $H^1$ to get
\begin{equation}\label{main}
\left|\int_x^\infty h(t,s)\,f(t,s)ds\right|\leq C\|u(1/t)\|_{H^1}\left( \left(1+\frac t{x^2}\right)\|f(t)\|_{L^\infty(x,\infty)}+\frac tx\,\|\partial_sf(t)\|_{L^2(x,1)}\right),
\end{equation}
which proves the bound in \eqref{nintT}.

Now we suppose that the lemma holds for all $1\leq k\leq n-1$ and we shall prove it for $n$. We denote
$$f_{n}(t,x)=\int_x^\infty h(t,s_1)\int_{s_1}^\infty h(t,s_2) ...\int_{s_{n-1}}^\infty h(t,s_n)\,f(t,s_n)\,ds_n...ds_1\quad,\quad f_0(t,x)=f(t,x).$$
In particular
$$f_n(t,x)=\int_x^\infty h(t,s_1)\, f_{n-1}(t,s_1)\,ds_1.$$
We perform now an integration by parts in the variable $s_1$ on $[x,1]$,
$$f_n(t,x)=\int_1^\infty  h(t,s)\,f_{n-1}(t,s_1)\,ds_1+2\sqrt{t}e^{-i\frac{1}{4t}-i\Phi(t,1)}\,u\left(\frac 1t,\frac 1t\right)\,f_{n-1}(t,1) $$
$$-\frac{2\sqrt{t}}{x}e^{-i\frac{x^2}{4t}-i\Phi(t,x)}\,u\left(\frac 1t,\frac xt\right)\,f_{n-1}(t,x)-\int_x^1u\left(\frac 1t,\frac st\right)\left(\frac{2\sqrt{t}}{s}e^{-i\frac{s^2}{4t}-i\Phi(t,s)}\,f_{n-1}(t,s)\right)_s.$$
It follows by Cauchy-Schwarz that
$$|f_n(t,x)|\leq C\|u(1/t)\|_{H^1}\left(\left(1+\frac{t}{x^2}\right)\|f_{n-1}(t)\|_{L^\infty(x,\infty)}+\frac tx\,\|\partial_sf_{n-1}(t)\|_{L^2(x,1)}\right).$$
Since again by Cauchy-Schwarz
$$\|\partial_sf_{n-1}(t,s)\|_{L^2(x,1)}=\|h(t,s)f_{n-2}(t,s)\|_{L^2(x,1)}\leq C\|u(1/t)\|_{H^1}\,\frac{1}{x}\,\|f_{n-2}(t)\|_{L^\infty(x,\infty)},$$
we get 
$$|f_n(t,x)|\leq C\|u(1/t)\|_{H^1}\left(1+\frac{t}{x^2}\right)\|f_{n-1}(t)\|_{L^\infty(x,\infty)}+C\|u(1/t)\|_{H^1}^2\frac t{x^2}\,\|f_{n-2}(t)\|_{L^\infty(x,\infty)}.$$
From the recursion hypothesis we have for all $s\geq x$ and $1\leq k\leq n-1$
$$|f_{k}(t,s)|\leq C^k\|u(1/t)\|_{H^1}^k \left(1+\frac t{x^2}\right)^{k-1}\left(\left(1+\frac t{x^2}\right)\|f(t)\|_{L^\infty(x,\infty)}+\frac tx\,\|\partial_sf(t)\|_{L^2(\min\{x,1\},1)}\right)$$
and the lemma follows.
\end{proof}

We are now able to iterate formulas \eqref{Tasy} and \eqref{Nasy} as follows.

\begin{lemma}\label{iteration}We set $a_1(t,x)=T^\infty$, $a_2(t,x)=-\Im N^{+\infty} \int_x^{\infty} h(t,s)ds$ and for $k\geq 1$ we define
$a_{2k+1}(t,x)$ by
$$(-1)^k\Re\int_x^\infty h(t,s_1)\int_{s_1}^\infty \overline{h(t,s_2)}...\Re\int_{s_{2k-2}}^\infty h(t,s_{2k-1})\int_{s_{2k-1}}^\infty  \overline{h(t,s_{2k})}T^\infty \,ds_{2k}...ds_1,$$
and $a_{2k+2}(t,x)$ by
$$(-1)^{k+1}\Re\int_x^\infty h(t,s_1)\int_{s_1}^\infty \overline{h(t,s_2)}...\Re\int_{s_{2k-2}}^\infty h(t,s_{2k-1})\int_{s_{2k-1}}^\infty  \overline{h(t,s_{2k})} \Im N^\infty\int_{s_{2k}}^\infty h(t,s_{2k+1})\,ds_{2k+1}...ds_1.$$
Then, there exists a constant $C>0$ such that for all $n\in \mathbb N^*$, $0<t\leq 1$ and $0<x$ the following decomposition holds
\begin{equation}\label{Tangsum}
T(t,x)=\sum_{j=1}^{2n}a_j(t,x) +b_n(t,x)\end{equation}
with
$$|b_n(t,x)|\leq C^{2n}\|u(1/t)\|_{H^1}^{2n} (1+a+\|u(1/t)\|_{L^2})\left(1+\frac t{x^2}\right)^{2n}$$
$$+C_5\sum_{k=0}^{n-1}C^{2k}\|u(1/t)\|_{H^1}^{2k}\left(1+\frac t{x^2}\right)^{2k}\,\left(\sqrt t+\frac{\sqrt t}{x}+\frac{t^2}{x^4}\right),$$
and
$$C_5=C(a+a^2+(1+a^2)C_0+(1+a^2)C_0^2+C_0^3),\qquad C_0=\|u(1)\|_{X^{\gamma^+}}+\|\partial_xu(1)\|_{X^{\gamma^+}}.$$
\end{lemma}

\begin{proof}We start by proving the lemma for $n=1$. 
Formulas \eqref{Tasy} and \eqref{Nasy} obtained in Proposition \ref{propas} give $$T(t,x)=T^{+\infty} -\Im N^{+\infty} \int_x^{\infty} h(t,s)ds-\Re\int_x^{\infty} h(t,s)\int_s^\infty \overline{h(t,s')}T(t,s')ds'ds+b_0(t,x),$$
with
$$b_0(t,x)=c_0(t,x)+\Im\int_x^{\pm\infty} h(t,s)\,d_0(t,s)\,ds.$$
Therefore 
$$b_1(t,x)=- \Re\int_x^\infty h(t,s_1)\int_{s_1}^\infty \overline{h(t,s_2)}\,T(t,s_2)\,ds_2+b_0(t,x).$$ 
We use Lemma \ref{restest} for $f=T$ and the fact that $T_s=\Re\overline{\psi}N$ to get
$$\left|\Re\int_x^\infty h(t,s_1)\int_{s_1}^\infty \overline{h(t,s_2)}\,T(t,s_2)\,ds_2\right|\leq C^2\|u(1/t)\|_{H^1}^2\left(1+\frac t{x^2}\right)^2\left(1+\left\|a+u\left(\frac1t,\frac xt\right)\right\|_{L^2(\min\{x,1\},1)}\right)$$
$$\leq C^2\|u(1/t)\|_{H^1}^2(1+a+\|u(1/t)\|_{L^2})\left(1+\frac t{x^2}\right)^2.$$
We are left with estimating $b_0(t,x)$. We deduce from Lemma 4.4 in \cite{BV3} that for some positive constant $\tilde C$,
\begin{equation}\label{boundb0}
|b_0(t,x)|\leq (C_1+\tilde CC_0(C_2+
a+a^2+C_0+C_0^2))\left(\sqrt t+\frac{\sqrt t}{x}+\frac{t^2}{x^4}\right),
\end{equation}
therefore we obtain the lemma for $n=1$.

For $n\geq 2$ we note that $b_n$ can be expressed as
\begin{equation}\label{recb}
b_n(t,x)=\sum_{k=1}^{n-1}(-1)^{k+1}\Re\int_x^\infty h(t,s_1)\int_{s_1}^\infty \overline{h(t,s_2)}...\int_{s_{2k-1}}^\infty \overline{h(t,s_{2k})}\,b_0(t,s_{2k})\,ds_{2k}...ds_1
\end{equation}
$$+(-1)^n\Re\int_x^\infty h(t,s_1)\int_{s_1}^\infty \overline{h(t,s_2)}...\Re\int_{s_{2n-2}}^\infty h(t,s_{2n-1})\int_{s_{2n-1}}^\infty  \overline{h(t,s_{2n})}T(t,s_{2n}) \,ds_{2n}...ds_1.$$
For the second multiple integral we use again Lemma \ref{restest} for $f=T$ and the fact that $T_s=\Re\overline{\psi}N$ to get the upper bound
$$C^{2n}\|u(1/t)\|_{H^1}^{2n}(1+a+\|u(1/t)\|_{L^2}) \left(1+\frac t{x^2}\right)^{2n}.$$
For the first integral we shall use Lemma \ref{restest} with $f=b_0$:
\begin{equation}\label{est1}
\left|(-1)^{k+1}\Re\int_x^\infty h(t,s_1)\int_{s_1}^\infty \overline{h(t,s_2)}...\int_{s_{2k-1}}^\infty \overline{h(t,s_{2k})}\,b_0(t,s_{2k})\,ds_{2k}...ds_1\right|
\end{equation}
$$\leq C^{2k}\|u(1/t)\|_{H^1}^{2k} \left(1+\frac t{x^2}\right)^{2k-1}\left( \left(1+\frac t{x^2}\right)\|b_0(t)\|_{L^\infty(x,\infty)}+\frac tx\,\|\partial_sb_0(t)\|_{L^2(\min\{x,1\},1)}\right).$$
 We have already the $L^\infty$ bound \eqref{boundb0} on $b_0$. 
Since integrating by parts  in the space variable from the quadratic phase of $h$ the integral in \eqref{Tasy}, one gets (see for instance \cite{BV3}, page 10) 
$$c_0(t,x)=-\Re\frac{2t}{-ix}\,\overline \psi(t,x)N (t,x)-\Re\int_x^\infty \frac{2t}{is^2}\,\overline\psi(t,s)N (t,s)\,ds-\Re\int_x^\infty \frac{2 t}{-is}\,\overline\psi(t,s) N_s (t,s)\,ds,$$
it is easy to see that 
$$\frac tx\,\left\|\partial_sc_0(t)\right\|_{L^2(\min\{x,1\},1)}\leq \left(a+a^2+\|u(1/t)\|_{H^1}+\|u(1/t)\|_{H^1}^2\right)\left(\frac{\sqrt{t}}{x}+\frac t{x^2}+\frac{t\sqrt{t}}{x^3}\right).$$
So we get by Cauchy-Schwarz and \eqref{c0d0}
\begin{equation}\label{est2}
\frac tx\|\partial_s b_0(t)\|_{L^2(\min\{x,1\},1)}\leq \frac tx\,\left\|\partial_sc_0(t)\right\|_{L^2(\min\{x,1\},1)}+\frac {t}{x^2}\|u(1/t)\|_{H^1}\|d_0(t)\|_{L^\infty(\min\{x,1\},1)}\end{equation}
$$\leq C\left(a+a^2+\|u(1/t)\|_{H^1}+\|u(1/t)\|_{H^1}^2+\|u(1/t)\|_{H^1}C_2\right)\left(\sqrt t+\frac{\sqrt{t}}{x}+\frac {t^2}{x^4}\right).$$
Therefore estimates \eqref{est1}, \eqref{boundb0} and \eqref{est2} and
$$\|u(1/t)\|_{H^1}\leq C(\|u(1)\|_{X^{\gamma^+}}+\|\partial_xu(1)\|_{X^{\gamma^+}})=C_0,$$
yield
$$\left|(-1)^{k+1}\Re\int_x^\infty h(t,s_1)\int_{s_1}^\infty \overline{h(t,s_2)}...\int_{s_{2k-1}}^\infty \overline{h(t,s_{2k})}\,b_0(t,s_{2k})\,ds_{2k}...ds_1\right|$$
$$\leq C^{2k}\|u(1/t)\|_{H^1}^{2k} \left(1+\frac t{x^2}\right)^{2k}\,C_5\left(\sqrt t+\frac{\sqrt{t}}{x}+\frac {t^2}{x^4}\right),$$
so the lemma follows.
\end{proof}

Our next aim is to replace each occurence of $h(t,s)$ in $a_j(t,x)$ by a function independent of time $\tilde h(s)$,
\begin{equation}
\label{tildeh}\tilde{h}(x)=i\widehat{f_+}\left(\frac x2\right)  \,e^{-ia^2\log |x|},
\end{equation}
up to getting a small error term. This will lead us to eventually identify the limit of $T(t,x)$ when $t$ goes to $0$. 
\begin{lemma}\label{lemmaansatzprev}
There exists a constant $\tilde C>0$ such that for all $g\in L^\infty$ with $g_s\in L^1\cap L^2$, $0<x\leq\tilde x$,
\begin{equation}\label{ansatz}
\left|\int_x^{\tilde x} (h(t,s)-\tilde h(s)) g(s)\,ds\right|\leq \tilde C_6(\|g\|_{L^\infty(x,\infty)}+\|g_s\|_{L^1(x,\infty)})\left(\frac{\sqrt{t}}{x}+t^{\frac 16^-}\right),
\end{equation}
with
$$\tilde C_6=\tilde C(\|u_1\|_{X^{\gamma^+}}+\|\partial_x u_1\|_{X^{\gamma^+}}+\|\partial_x^2u_1\|_{X^{\gamma^+}}).$$
\end{lemma}
The proof of this result goes exactly like the one of Lemma 4.3 in \cite{BV3} because in that one we did not use the weight condition. In fact, note that the asymptotic state $f_+$ does not belong in general to weighted spaces.

\begin{lemma}\label{ajest}There exists a constant $C>0$ such that for all $n\in \mathbb N$, $0<t\leq 1$ and $0<x$ we have
\begin{equation}\label{replace}
\left|\int_x^\infty h(t,s_1)\int_{s_1}^\infty h(t,s_2) ...\int_{s_{n-1}}^\infty h(t,s_n)\,ds_n...ds_1-\int_x^\infty \tilde h(s_1)\int_{s_1}^\infty \tilde h(s_2) ...\int_{s_{n-1}}^\infty \tilde h(s_n)\,ds_n...ds_1\right|
\end{equation}
$$\leq C^n_6\left(1+\frac t{x^2}\right)^{n-1}\left(1+\frac 1x\right)\left(\frac{\sqrt{t}}{x}+t^{\frac 16^-}\right),$$
with
$$C_6=C(\|u_1\|_{X^{\gamma^+}}+\|\partial_x u_1\|_{X^{\gamma^+}}+\|\partial_x^2u_1\|_{X^{\gamma^+}}).$$
Moreover, any ocurrence of $h$ can be replaced by $\overline h$, provided that the correspondent $\tilde h$ is replaced by $\overline{\tilde h}$. 
\end{lemma}

\begin{proof}

Lemma \ref{lemmaansatzprev} with $g(s)=1$ gives the Lemma for $n=1$, by choosing $C>\tilde C$. For $n\geq 2$ we shall proceed by recursion, with a constant $C$ 
\begin{equation}\label{condC}
C>\max\{1,\tilde C, 3c\,C(a)+\hat C\},
\end{equation}
where $\hat C$ stand from the constant of Lemma \ref{restest}, $c$ stands for the constant in Cauchy-Schwarz inequality and $C(a)$ is used in 
$$\|f_+\|_{H^1}\leq C(a)(\|u_1\|_{X^{\gamma^+}}+\|\partial_x u_1\|_{X^{\gamma^+}}).$$
We write the difference in \eqref{replace} as
$$\int_x^\infty \tilde h(s_1)\left(\int_{s_1}^\infty h(t,s_2) ...\int_{s_{n-1}}^\infty h(t,s_n)\,ds_n...ds_1-\int_{s_1}^\infty \tilde h(s_2) ...\int_{s_{n-1}}^\infty \tilde h(s_n)\,ds_n...ds_1\right)$$
$$+\int_x^\infty \left(h(t,s_1)-\tilde h(s_1)\right)\int_{s_1}^\infty h(t,s_2) ...\int_{s_{n-1}}^\infty h(t,s_n)\,ds_n...ds_1=I_1(t,x)+I_2(t,x).$$
By the recursion hypothesis
$$|I_1(t,x)|\leq \|\widehat{f_+}\|_{L^1}C^{n-1}_6\left(1+\frac t{x^2}\right)^{n-2}\left(1+\frac 1x\right)\left(\frac{\sqrt{t}}{x}+t^{\frac 16^-}\right).$$
$$\leq \frac{c\,C(a)}{C}\,C^{n}_6\left(1+\frac t{x^2}\right)^{n-2}\left(1+\frac 1x\right)\left(\frac{\sqrt{t}}{x}+t^{\frac 16^-}\right).$$

Let us denote now 
$$f_{n}(t,x)=\int_x^\infty h(t,s_1)\int_{s_1}^\infty h(t,s_2) ...\int_{s_{n-1}}^\infty h(t,s_n)\,ds_n...ds_1,\quad f_0(t,x)=1.$$
We apply \eqref{ansatz} to get
\begin{equation}\label{I2}
|I_2(t,x)|=\left|\int_x^{\tilde \infty} (h(t,s_1)-\tilde h(s_1)) f_{n-1}(t,s_1)\,ds_1\right|
\end{equation}
$$\leq C_6(\|f_{n-1}(t)\|_{L^\infty(x,\infty)}+\|\partial_s f_{n-1}(t)\|_{L^1(x,\infty)})\left(\frac{\sqrt{t}}{x}+t^{\frac 16^-}\right).$$
Lemma \ref{restest} with $f(t,s)=1$ gives us for $k\geq 1$
\begin{equation}\label{fninfty}
\|f_k(t)\|_{L^\infty(x,\infty)}\leq \hat C^k\|u(1/t)\|_{H^1}^k\left(1+\frac t{x^2}\right)^k.
\end{equation}
Since
$$|\partial_s f_{n-1}(t,s)|=|h(t,s)f_{n-2}(t,s)|=\frac{2}{s\sqrt{t}}\left|(u_s)\left(\frac 1t,\frac st\right)\right||f_{n-2}(t,s)|,$$
we obtain by Cauchy-Schwarz
\begin{equation}\label{fnLp}
\|\partial_sf_{n-1}(t)\|_{L^1(x,\infty)}\leq 2c\,\hat C^{n-2}\|u(1/t)\|_{H^1}^{n-1}\,\frac 1x\,\left(1+\frac t{x^2}\right)^{n-2}.
\end{equation}
Therefore by using \eqref{fninfty} and \eqref{fnLp} in \eqref{I2} we obtain
$$|I_2(t,x)|\leq \left(1+\frac{2c}{\hat C}\right) \hat C^{n-1}\|u(1/t)\|_{H^1}^{n-1}C_6\left(1+\frac t{x^2}\right)^{n-1}\left(1+\frac 1x\right)\left(\frac{\sqrt{t}}{x}+t^{\frac 16^-}\right),$$
$$\leq \left(1+\frac{2c}{\hat C}\right) \left(\frac{\hat C\,C(a)}{C}\right)^{n-1}C_6^n\left(1+\frac t{x^2}\right)^{n-1}\left(1+\frac 1x\right)\left(\frac{\sqrt{t}}{x}+t^{\frac 16^-}\right),$$
and the lemma follows also for $n\geq 2$ since choosing $C>\hat C$ such that \eqref{condC} holds implies
$$\frac {c\,C(a)}C+\left(1+\frac{2c}{\hat C}\right)\frac {\hat C\,C(a)}C<1.$$\\
\end{proof}

\begin{prop}\label{proptanglim}If $u(1)$ and its first two derivatives are small in $X^{\gamma^+}$ then for all $x>0$ there exists a limit of $T(t,x)$ at time $t=0$,
$$\underset{x\rightarrow\infty}{\lim}T(t,x)=T(0,x),$$
with the rate of convergence
$$T(t,x)-T(0,x)=\mathcal O(t^{\frac 16^-}).$$
Moreover, for $t\leq x^2$ 
$$|T(t,x)-T(0,x)|\leq C_7(t,x),$$
where $C_7(t,x)$ is a linear combination of powers $\left(\frac{\sqrt {t}}{x}\right)^s,1\leq s\leq 4$.
\end{prop}
\begin{proof}We fix $0<x$. Let $\tilde a_1(x)=T^\infty$, $\tilde a_2(x)=-\Im N^{+\infty} \int_x^{\infty} \tilde h(s)ds$ and for $k\geq 1$ set $\tilde a_{2k+1}(x)$ to be
\begin{equation}
\label{tildea} (-1)^k\Re\int_x^\infty \tilde h(s_1)\int_{s_1}^\infty \overline{\tilde h(s_2)}...\Re\int_{s_{2k-2}}^\infty \tilde h(s_{2k-1})\int_{s_{2k-1}}^\infty  \overline{\tilde h(s_{2k})}T^\infty \,ds_{2k}...ds_1,
\end{equation}
and $\tilde a_{2k+2}(x)$ to be 
\begin{equation}
\label{tildeaa}
(-1)^{k+1}\Re\int_x^\infty \tilde h(s_1)\int_{s_1}^\infty \overline{\tilde h(s_2)}...\Re\int_{s_{2k-2}}^\infty \tilde{h}(s_{2k-1})\int_{s_{2k-1}}^\infty  \overline{\tilde h(s_{2k-2})} \Im N^\infty\int_{s_{2k}}^\infty \tilde h(s_{2k+1})\,ds_{2k+1}...ds_1.
\end {equation}
Gathering  Lemma \ref{iteration} and Lemma \ref{ajest}  we can decompose
\begin{equation}\label{Tangsum}
T(t,x)=\sum_{j=1}^{2n}\tilde a_j(x) +R_n(t,x)
\end{equation}
with
$$ |R_n(t,x)|\leq C^{2n}\|u(1/t)\|_{H^1}^{2n}(1+a+\|u(1/t)\|_{L^2}) \left(1+\frac t{x^2}\right)^{2n}$$
$$+C_5\sum_{k=0}^{n-1}C^{2k}\|u(1/t)\|_{H^1}^{2k}\left(1+\frac t{x^2}\right)^{2k}\,\left(\sqrt t+\frac{\sqrt t}{x}+\frac{t^2}{x^4}\right)$$
$$+\sum_{j=1}^{2n-1}C^j_6\left(1+\frac t{x^2}\right)^{j-1}\left(1+\frac 1{x}\right)\left(\frac{\sqrt{t}}{x}+t^{\frac 16^-}\right).$$
Recall that
$$\|u(1/t)\|_{H^1}\leq C(a)(\|u(1)\|_{X^{\gamma^+}}+\|\partial_x u(1)\|_{X^{\gamma^+}}),$$
and
$$C_6=C(\|u(1)\|_{X^{\gamma^+}}+\|\partial_x u(1)\|_{X^{\gamma^+}}+\|\partial_x^2u(1)\|_{X^{\gamma^+}}).$$
Let $t\leq x^2$. 
Hence, $u(1)$ and its first two derivatives are small in $X^{\gamma^+}$ it follows that there exists $n(t,x)$ large enough such that for all $n\geq n(t,x)$
$$ |R_{n}(t,x)|\leq 2C_5\left(\sqrt t+\frac{\sqrt t}{x}+\frac{t\sqrt t}{x^3}\right)+2\left(1+\frac 1{x}\right)\left(\frac{\sqrt{t}}{x}+t^{\frac 16^-}\right).$$
Note also that 
$$|\tilde a_j(x)|\leq C^{2j}\|f_+\|_{H^1}^{2j}\leq C^{2j}C(a)^{2j}\left(\|u(1)\|_{X^{\gamma^+}}+\|\partial_xu(1)\|_{X^{\gamma^+}}\right)^{2j}.$$
Therefore, there exists $\tilde n(t,x)\geq n(t,x)$ large enough such that
$$ |R_{\tilde n(t,x)}(t,x)|+\left |\sum_{j=2\tilde n(t,x)}^{\infty}\tilde a_j(x)\right|\leq 3C_5\left(\sqrt t+\frac{\sqrt t}{x}+\frac{t\sqrt t}{x^3}\right)+3\left(1+\frac 1{x}\right)\left(\frac{\sqrt{t}}{x}+t^{\frac 16^-}\right).$$
From \eqref{Tangsum} we get then 
$$T(t,x)-\sum_{j=1}^\infty \tilde a_j(x)=\mathcal O(t^{\frac 16^-}),$$
and the first part of the Proposition follows taking $T(0,x)=\sum_{j=1}^\infty \tilde a_j(x) $ and by  letting $t$ go to $0$.

Now we shall get some extra-information on $T(0,x)$.  
Note that in view of the definition of the $\tilde a_j$ \eqref{tildea}-\eqref{tildeaa} we deduce that for $f_+$ small enough in $H^1$, 
$$\left\|T(0)\right\|_{L^\infty}\leq \sum_{j=1}^\infty C^{2j}\|f_+\|_{H^1}^{2j}<\infty,\quad \left\|T_x(0)\right\|_{L^1}+\left\|T_x(0)\right\|_{L^2}\leq \|f_+\|_{H^1}\sum_{j=1}^\infty C^{2j}\|f_+\|_{H^1}^{2j}<\infty.$$
We get also that for fixed $x>0$ the complex vector $\tilde N(t,x)$ has a limit as $t=0$ in the following way. From Proposition \ref{propas} we have
$$\tilde N(t,x)-N^\infty-i\int_x^{\infty} \overline{h(t,s)}\,T(t,s)\,ds=\mathcal O(\sqrt {t}).$$
We have proved above that $T(t,x)=T(0,x)+\mathcal O(t^{\frac 16^-})$. Take $\epsilon>0$ and recall the bound \eqref{x>1}. Then there exists $M_\epsilon$ large enough such that for $M\geq M_\epsilon$
$$|\tilde N(t,x)-N^\infty-i\int_x^{M} \overline{h(t,s)}\,T(0,s)\,ds|\leq \mathcal O(t^{\frac 16^-})+\epsilon.$$
Secondly, $T(0)\in L^\infty$  and $T_s(0)\in L^1$, so we can apply formula \eqref{ansatz} with $g(s)=T(0,s)$ to obtain
$$|\tilde N(t,x)-N^\infty-i\int_x^{M} \overline{\tilde h(s)}\,T(0,s)\,ds|\leq \mathcal O(t^{\frac 16^-})+\epsilon.$$
Now, since $\tilde h\in L^1$ and $T(0)\in L^\infty$, by choosing $M$ large enough we get 
$$|\tilde N(t,x)-N^\infty-i\int_x^{\infty} \overline{\tilde h(s)}\,T(0,s)\,ds|\leq \mathcal O(t^{\frac 16^-})+\epsilon.$$
As a conclusion, there exists a limit of $\tilde N(t,x)$ as $t$ tends to $0$ and
\begin{equation}\label{eqN0}
\tilde N(0,x)=N^\infty+i\int_x^\infty \overline{\tilde h(s)} T(0,s)\,ds.
\end{equation}
So in particular $\tilde N(0)\in L^\infty$ and $\tilde N_x(0)\in L^1\cap L^2$, and we can argue similarly for $T$ to get
\begin{equation}\label{eqT00}
T(0,x)=T^\infty-\Im \int_x^\infty\tilde h(s) \tilde N(0,s)\,ds.
\end{equation}
Gathering \eqref{eqN0} and \eqref{eqT00} we obtain an integral equation for $T(0)$ :
\begin{equation}\label{eqT0int}
T(0,x)=T^\infty-\Im \int_x^\infty\tilde h(s) N^\infty\/ds-\Re\int_x^\infty\tilde h(s)\int_s^\infty \overline{\tilde h(s')} T(0,s')\,ds'\,ds.
\end{equation}
Then the last estimate of the Proposition follows as in the proof of Proposition 4.6 in \cite{BV3} by using \eqref{Tasy}-\eqref{Nasy}-\eqref{c0d0}, Lemma \ref{restest} and Lemma \ref{lemmaansatzprev}.

\end{proof}

\subsection{Properties of the trace at time $t=0$}\label{sectT0bis}
From \eqref{eqN0} and \eqref{eqT00} we obtain
$$T_x(0,x)=\Im\left(\tilde h(x)\,\tilde N(0,x)\right),\quad \tilde N_x(0,x)=-i\overline{\tilde h(x)}\,T(0,x).$$
We recall now that
$$\tilde{h}(x)=i\widehat{f_+}\left(\frac x2\right)  \,e^{-ia^2 \log |x|}.$$
As a conclusion we have that $(T(0),\tilde N(0))$ satisfy 
\begin{equation}\label{eqT0}\left\{\begin{array}{c}T_x(0,x)=\Re\left(\widehat{f_+}\left(\frac x2\right)\,e^{-ia^2\log |x|}\,\tilde N(0,x)\right),\\ \,\\
\tilde N_x(0,x)=-\overline{\widehat{f_+}\left(\frac x2\right)\,e^{-ia^2\log |x|}}\,T(0,x).
\end{array}\right.\end{equation}
Moreover, it was shown in \S5 of \cite{BV3}, without using any weight condition, that there exists a rotation $R$ such that
$$RT(0,0^\pm)=A^\pm\quad,\quad R\tilde N(0,0^\pm)=B^\pm.$$
Finally, since
$$f_+=\mathcal F^{-1}\left(g(2\cdot) e^{ia^2\log |2\cdot|}\right),$$ 
it follows that $\tilde h(x)=ig(x)$ so the traces $(T_0,\tilde N_0)$ given in \eqref{systinitial} and $(RT(0),R\tilde N(0))$ coincide, since they are both solutions to \eqref{eqT0} with same initial value. We recall that we have constructed $\chi(t)$ with $\chi(0,0)=\chi_0(0)=0$. Now we have obtained also that $R\partial_s\chi(0,x)=RT(0,x)=T_0(x)$ so we get that $R\chi$ has trace $\chi_0$.
Therefore we have constructed the solution of the statement of Theorem \ref{ivp} for positive times.

\subsection{Continuation through time $t=0$}\label{sectext}
We denote $\chi_0^*(x)=\chi(0,-x)$. Then $$T_0^*(x)=-T(0,-x),$$ and we define $g^*$ and $N^*_0(x)$ given by the system
\begin{equation}\label{systbis}
\left\{\begin{array}{c}T^*_{0,x}(x)=\Re g^*(x)\tilde N_0^*(x),\\ \tilde N^*_{0,x}(x)=-\overline{g^*}(x)T_0^*(x),
\end{array}\right.\end{equation}
with initial data $\rho(A^+,B^+)$ for $x>0$ and $\rho(A^-,B^-)$ for $x<0$. Note that in view of Proposition \ref{contssim}, the initial data for $T^*_0$ makes sense. Also in view of Proposition \ref{contssim}, $\rho A^\pm =R^\mp (-A^\mp)$ and $\rho B^\pm=R^\mp \overline{B^\mp}$ so the solution of \eqref{systbis} can be expressed in terms of the initial one \eqref{systinitial},
$$N^*_0(x)=R^\mp\overline{N_0(-x)},\quad g^*(x)=\overline{g(-x)}.$$
In particular, $g^*$ satisfies the same conditions as $g$, and we can apply Theorem \ref{ivp} for positive times and initial data $\chi_0^*(x)=\chi(0,-x)$. This yields $\chi^*(t,x)$ solution of the binormal flow  for positive times obtained with initial data $\chi(0,-x)$. Then for negative times we shall extend $\chi(t,x)$ by 
$$\chi(t,x)=\chi^*(-t,-x)$$ 
and obtain the solution in Theorem \ref{ivp} on $[-1,1]$.

\subsection{Uniqueness of the solution}
For proving the uniqueness, suppose that there exists another solution $\chi^\star$ of the binormal flow on positive times, such that $\chi^\star(0)=\chi_0$, with the following regularity. We suppose $\chi^\star\in\mathcal C([-1,1],Lip)\cap \mathcal C([-1,1]\backslash \{0\},\mathcal C^4)$ and assume that its filament functions associated at times $\pm1$ are of the type $(a+u^\star(\pm1,x))e^{i\frac{x^2}{4}}$ with $u^\star(\pm1)$ small in $X^\gamma_1\cap H^4$ with respect to $a$ for some $0<\gamma<\frac 14$.  We shall prove that $\chi^\star=\chi$ on positive time (the same argument applied to $\chi(-t,-x)$ and $\chi_0(-x)$ will show the uniqueness for negative times). 

In view of the scattering result from \cite{BV2} we denote $u^\star(t,x)$ to be the solution of \eqref{NLS} with initial data at time $t=1$ the function $u^\star(1)$, and we denote by $f_{+}^\star$ its asymptotic state.  
Since 
\begin{equation}\label{estbin}|\chi^\star_x\land \chi^\star_{xx}(\tau,x)|=|c^\star b^\star(\tau,x)|=\frac1{\sqrt \tau} |(a+\overline{u^\star})(\frac1\tau,\frac x\tau)|\leq \frac{C}{\sqrt \tau},
\end{equation}
we can write the Duhamel formula from time $0$ to a positive time $t$
\begin{equation}\label{Duhbin}
\chi^\star(t,x)=\chi_0(x)+\int_0^t \chi^\star_x\land \chi^\star_{xx}(\tau,x) d\tau.
\end{equation}
Differentiating in $x$ we obtain a formulation for the tangent vector, at $x\neq 0$, 
\begin{equation}\label{DuhbinT}
T^\star(t,x)=T_0(x)+\partial_x\int_0^t \chi^\star_x\land \chi^\star_{xx} (\tau,x)d\tau.
\end{equation}

We define $g^\star$ by $f_{+}^\star=\mathcal F^{-1}\left(g^\star(2\cdot)e^{ia^2\log |2\cdot|}\right).$ From \cite{BV3}, together with the computations in the previous subsections \S 3.1-\S 3.3 for avoiding weighted conditions on the data, we have that the tangent vector $T^\star(t,x)$ has a limit $T^\star(0,x)$ as $t$ goes to zero, for $x\neq 0$. We also get the  existence of a $\mathbb C^3$-valued function $\tilde N^\star$ orthonormal on $T^\star(0)$ such that for $x>0$ 
\begin{equation}\label{systinitialbis}
\left\{\begin{array}{c}T_{x}^\star(0,x)=\Re (g^\star(x)\tilde N^\star(x)),\\ \tilde N_{x}^\star(x)=-\overline{g^\star}(x)T^\star(0,x),
\end{array}\right.\end{equation}
with initial data $(A^+,B^+)$ (and similar for $x<0$). Since $T^\star$ is a solution of the Schr\"odinger map \eqref{schmap} we can write for $0<\tilde t<t$,
$$T^\star(t,x)=T^\star(\tilde t,x)+\int_{\tilde t}^t T^\star\land T^\star_{xx}(\tau,x) d\tau=T^\star(\tilde t,x)+\partial_x\int_{\tilde t}^t \chi^\star_x\land \chi^\star_{xx}(\tau,x) d\tau.$$
Let $\phi$ be a compactly supported test function away from $x=0$, such that $\phi'\in L^1$. Then by integrating by parts,
$$\int T^\star(t,x)\phi(x)\,dx=\int T^\star(\tilde t,x)\phi(x)\,dx-\int\int_{\tilde t}^t \chi^\star_x\land \chi^\star_{xx}(\tau,x) d\tau\phi'(x)\,dx.$$
Then, in view of \eqref{DuhbinT} we obtain
$$\int (T_0(x)-T^\star(0,x))\phi(x)dx=\int(T^\star (\tilde t,x)-T^\star (0,x))\phi(x)\,dx+\int \int_0^{\tilde t} \chi^\star_x\land \chi^\star_{xx}(\tau,x) d\tau\,\phi'(x) dx.$$
Using Proposition \ref{proptanglim} and \eqref{estbin} we obtain
$$\int (T_0(x)-T^\star(0,x))\phi(x)dx=0.$$
Since $\partial_xT^\star(0,x)\in L^1\cap L^2$ we obtain that $T^\star(0,x)$ is continuous for $x\neq 0$. The same is valid for $T_0(x)$. 
By taking $\phi$ approximating the Dirac distribution located at $x\neq 0$ we obtain that $ T^\star(0,x)=T_0(x)$. 
Therefore using \eqref{systinitialbis} and \eqref{systinitial} we get
$$(\Re\tilde N_{0}-\Re N^\star)(x)=\int_0^x-(\Re g-\Re g^\star)(y)T_0(y)=-\int_0^x\langle T_{0x},(\Re\tilde N_{0}-\Re N^\star)\rangle T_0(y)dy$$
so by using again \eqref{systinitial} and Cauchy-Schwarz inequality
$$|(\Re\tilde N_{0}-\Re N^\star)(x)|\leq C\int_0^x |g(y)||(\Re\tilde N_{0}-\Re N^\star)(y)|dy\leq C\|f_+\|_{L^2}\left(\int_0^x |(\Re\tilde N_{0}-\Re N^\star)(y)|^2dy\right)^\frac 12.$$
We conclude by Gronwall that $\Re\tilde N_0=\Re N^\star$. Similary we obtain $\Im\tilde N_0=\Im N^\star$, so $\tilde N_0=N^\star$. Therefore we get $g=g^\star,$ and implicitly $f_+=f_{+}^\star$. From \cite{BV2} we have uniqueness of the wave operators, so $u(t,x)=u^\star(t,x)$. Therefore curvature and torsion are the same for $\chi^\star$ and $\chi$, so $\chi^\star$ and $\chi$ are the same modulo one rotation (due to the choice of an initial data for the Frenet frame when integrating the Frenet system to obtain the Frenet frame) and one translation (due to the choice of the location of the curve when integrating the binormal flow to obtain the binormal solution). Since the initial data $\chi^\star(0)$ and $\chi(0)$ coincide as oriented curves, it follows that $\chi^\star$ and $\chi$ coincide.

\subsection{Properties of the solution}

Recall that by using the Frenet system, the binormal flow \eqref{binormal} can be written as
$$\chi_t=c\,b.$$
The estimate \eqref{convbin} in part i) of Theorem \ref{ivp} follows from the fact that the solution $u(t)$ is small, so the curvature is close to the one of the sefsimilar solutions $c_a(t',x)=\frac{a}{\sqrt {t'}}$,
$$|\chi(0,x)-\chi_0(x)|\leq\left|\int_0^t \chi_{t'}(t',x)\,dt'\right|=\left|\int_0^t c(t',x)\,b(t',x)\,dt'\right|\leq\int_0^t c(t',x)\,dt'\leq C\sqrt{t}.$$

To prove ii) we shall use the notations in \S 2 of \cite{BV3} on parallel frames:
$$\chi_t=T\land T_x=T\land(\alpha e_1+\beta e_2)=\alpha e_2-\beta e_1=\Im(\overline{\psi} N),$$
so integrating between two fixed times $0<t_2\leq t_1\leq 1$ we obtain
$$\chi(t_1,x)-\chi(t_2,x)=\int_{t_2}^{t_1}\Im(\overline{\psi} N)(t,x)dt=\Im\int_{t_2}^{t_1}\frac{e^{-i\frac{x^2}{4t}}}{\sqrt t}(a+u)\left(\frac 1t,\frac xt\right) N(t,x)dt.$$
We perform an integration by parts from the oscillating phase,
$$\chi(t_1,x)-\chi(t_2,x)=\Im\left[\frac{4t^2}{ix^2}\frac{e^{-i\frac{x^2}{4t}}}{\sqrt t}(a+u)\left(\frac 1t,\frac xt\right) N(t,x)\right]_{t_2}^{t_1}$$
$$+\frac{1}{x^2}\,\Re\int_{t_2}^{t_1}e^{-i\frac{x^2}{4t}}\partial_t\left(4t\sqrt t(a+u)\left(\frac 1t,\frac xt\right) N(t,x)\right)dt.$$
Since $N_t=i\psi_xT-\frac{a^2-t|\psi|^2}{2t}N$ and $u$ solves \eqref{NLS} it follows that
$$|\chi(t_1,x)-\chi(t_2,x)|\leq C\frac{t_1\sqrt{t_1}}{x^2}(a+\|u(1/t)\|_{L^\infty([t_2,t_1],L^\infty)})+C\frac{\sqrt{t_1}}{x^2}\|\partial_x^2u(1/t)\|_{L^\infty([t_2,t_1],L^\infty)}$$
$$+C(a)\frac{1}{x^2\sqrt{t_1}}(\|u(1/t)\|_{L^\infty([t_2,t_1],L^\infty)}+\|u(1/t)\|_{L^\infty([t_2,t_1],L^\infty)}^3)+C\frac{\sqrt{t_1}}{x}\|\partial_xu(1/t)\|_{L^\infty([t_2,t_1],L^\infty)}$$
$$+C\frac{t_1}{x}(a+\|u(1/t)\|_{L^\infty([t_2,t_1],L^\infty)})^2+C\frac{t_1}{x}(a+\|u(1/t)\|_{L^\infty([t_2,t_1],L^\infty)})\|\partial_xu(1/t)\|_{L^\infty([t_2,t_1],L^\infty)}$$
$$+C(a)\frac{t_1\sqrt{t_1}}{x^2}(a+\|u(1/t)\|_{L^\infty([t_2,t_1],L^\infty)})(\|u(1/t)\|_{L^\infty([t_2,t_1],L^\infty)}+\|u(1/t)\|_{L^\infty([t_2,t_1],L^\infty)}^2).$$
In view of \eqref{unifbounds} the first asymptotic behaviour \eqref{convinfinity} in ii) follows. The second asymptotic behaviour in ii)  was proved in \S 3.2 of \cite{BV3}, and \eqref{convinfinityT} was displayed in Proposition \ref{propas}.

Let us prove (iii). On one hand, as done above for proving \eqref{convbin} we get
$$\left|\int_{-1}^1\int\chi_t(t,x)\phi(t,x)\,dx dt\right|\leq C\|\phi\|_{L^\infty L^1}\int_{-1}^1\frac{1}{\sqrt {|t|}}\,dt<\infty.$$
On the other hand $\chi$ is a strong solution of \eqref{binormal} on $[-1,0[\cap]0,1]$, so \eqref{binweak} of iii) follows also. 

Estimate \eqref{convT} comes from Proposition \ref{proptanglim}, and the fact that $\partial_xT(0)\in L^1\cap L^2$ was proved at the beginning of \S \ref{sectT0bis}, so the assertions in iv) are proved.

Finally, v) follows immediately from  \eqref{binweak}. The proof of Theorem \ref{ivp} is complete.

\section{Proof of Theorem \ref{cont}}\label{sectcont}
Concerning Theorem \ref{cont} we recall that its part concerning positive times $t\geq 0$ was the main result in \cite{BV3}, under the assumption that weighted conditions are satisfied by $u(1)$. In the proof of Theorem \ref{ivp} we have removed these conditions, provided that the initial data (or the asymptotic state) are small in spaces of type $\partial_x^kf\in X^\gamma$ for $0\leq k\leq 4$ and some $\gamma<\frac 14$. For extending $\chi$ to negative times, we proceed as explained above in \S \ref{sectext}.

\section{Appendix: Analysis in weighted spaces of the NLS equation}\label{appendix}
We recall that the results in this paper could have been obtained easily from \cite{BV2}, provided that if the initial state $u(1)$ of equation \eqref{NLS} is in weighted spaces implies its asymptotc state $f_+$ belongs also to weighted spaces and reciprocally. 
We shall present here a detailed analysis of the solutions of the linear part of \eqref{NLS} whose data belong to weighted spaces. As a conclusion, weighted spaces are not an appropriate setting for scattering theory of equation \eqref{NLS}, that actually is in contrast with the case of the classical linear Schr\"odinger equation. \\

We sset $\omega(t)=S(t,t_0)\omega(t_0)$ to be the solution of
\begin{equation}\label{lin} i\omega_t+\omega_{xx}+\frac{a^2}{2t}(\omega+\overline \omega)=0,
\end{equation}
with initial data $\omega(t_0)$ at time $t_0$. We set 
then $v(t)=J(t)\omega(t)=(x+2it\nabla)\omega(t)$, which satisfies
\begin{equation}\label{linJ}iv_t+v_{xx}+\frac{a^2}{2t}(v+\overline v)=\frac{a^2}{2t}(\overline{J\omega}-J\overline \omega)=-2ia^2\overline \omega_x,\end{equation}
with initial data $v(t_0)=J(t_0)\omega(t_0)$ at time $t_0$. 
Taking the real and imaginary parts, and then passing in Fourier variables yields
\begin{equation}\label{remodes}
\partial_t\,\,\widehat{\Re v}(t,\xi)=\xi^2\,\widehat{\Im v}(t,\xi)-2ia^2\xi\,\widehat{\Re \omega}(t,\xi),
\end{equation}
\begin{equation}\label{immodes}
\partial_t\,\,\widehat{\Im v}(t,\xi)=-\xi^2\,\widehat{\Re v}(t,\xi)+ \frac {a^2}{t}\,\widehat{\Re v}(t,\xi)+2ia^2\xi\,\widehat{\Im \omega}(t,\xi).
\end{equation}

In the rest of this Appendix we shall try to  understand how the $L^2$ norm of $v(t)=J(t)\omega(t)$ grows in time. First, recall that the classical linear Schr\"odinger equation commutes with $J(t)$.  We have the following result.

\begin{prop}\label{linJdescription}For $\xi\neq 0$ the evolution of $J(t)\omega(t)$ is described by 
$$J(t)\omega(t)=S(t,1)J(1)\omega(1)+S(t,1)\mathcal F^{-1}\left(\frac {2a^2}\xi\widehat{\Re\omega}(1,\xi)\right)-\mathcal F^{-1}\left(\frac {2a^2}\xi\widehat{\Re\omega}(t,\xi)\right).$$
\end{prop}
\begin{proof}
From \eqref{remodes} we deduce that we can write
\begin{equation}\label{writtingv}
\hat v=\widehat{\Re v}+i\frac{\partial_t\widehat{\Re v}}{\xi^2}-\frac{2a^2}{\xi}\widehat{\Re\omega}.
\end{equation}
Let us recall that since $\omega$ solves \eqref{lin}, it follows that
$$\partial_t\,\,\widehat{\Re \omega}(t,\xi)=\xi^2\,\widehat{\Im \omega}(t,\xi).$$
Gathering this with \eqref{remodes} and \eqref{immodes} we obtain that $\widehat{\Re v}(t,\xi)$ satisfies the the second order equation
\begin{equation}\label{secondeq}
\partial_t^2f(t,\xi)=\xi^2\left(-\xi^2+\frac {a^2}{t}\right)f(t,\xi),
\end{equation}
with initial data 
\begin{equation}\label{idv}f(1,\xi)=\widehat{\Re v(1)}(\xi),\quad \partial_t f(1,\xi)=\xi^2\widehat{\Im v(1)}(\xi)-2ia^2\xi\widehat{\Re\omega}(1,\xi).\end{equation}

Now we notice that if $f(t,\xi)$ is a solution of \eqref{secondeq} with an initial data that satisfy
\begin{equation}\label{condid}
f(1,\xi)=\overline{f(1,-\xi)},\quad \partial_tf(1,\xi)=\overline{\partial_t f(1,-\xi)},
\end{equation}
then this property is conserved in time, and we have the identity
$$f(t,\xi)+i\frac{\partial_t f(t,\xi)}{\xi^2}=\mathcal{F} S(t,1)\mathcal{F}^{-1}\left(f(1,\xi)+i\frac{\partial_t f(1,\xi)}{\xi^2}\right).$$
This is valid since both functions satisfy the equation of the Fourier transform of a $S(t,1)$ evolution (see \eqref{lin}):
\begin{equation}\label{linF}
iU_t(t,\xi)-\xi^2U(t,\xi)+\frac{a^2}{2t}(U(t,\xi)+\overline{U(t,-\xi)})=0,
\end{equation}
with the same initial data. 

Finally, since the initial condition \eqref{idv} satisfy the condition \eqref{condid} it follows that
$$\widehat{\Re v}+i\frac{\partial_t\widehat{\Re v}}{\xi^2}=\mathcal{F} S(t,1)\mathcal{F}^{-1}\left(\widehat{v(1)}(\xi)+\frac{2a^2}{\xi}\widehat{\Re\omega}(1,\xi)\right),$$
so in view of \eqref{writtingv} the proof is complete.\\
\end{proof}

So first we shall see in the next Proposition that imposing vanishing condition at the zero-modes on the asymptotic state we obtain that the solution belongs to weighted spaces. Let us recall the asymptotic results obtained  in \S 2.1 of \cite{BV2} for $\hat \omega$ with $\omega$ the solution of \eqref{lin}. For $4a^2\leq t\xi^2$ we denoted
\begin{equation}\label{descromega}\left(\begin{array}{c}\widehat{\Re w}\\ \widehat{\Im w}\end{array}\right)(t,\xi)=\left(\begin{array}{c}Y\\ Z\end{array}\right)(t\xi^2,\xi)
=\left(\begin{array}{cc} e^{i\Psi(t\xi^2)} &e^{-i\Psi(t\xi^2)}\\i\alpha(t\xi^2)e^{i\Psi(t\xi^2)}&-i\alpha(t\xi^2)e^{-i\Psi(t\xi^2)}\end{array}\right)\left(\begin{array}{c} Y_2\\ Z_2\end{array}\right)(t\xi^2,\xi),\end{equation}
with
$$\alpha(\tau)=\sqrt{1-\frac{a^2}{\tau}}\quad,\quad \Psi(\tau)=\tau-\frac{a^2}{2}\log \tau-\int_\tau^\infty \alpha(s)-1+\frac{a^2}{2s}\,ds.$$
We proved that for $\tau\leq 4a^2$
\begin{equation}\label{systyz}\partial_\tau\left(\begin{array}{c} Y_2\\ Z_2\end{array}\right)(\tau,\xi)=M(\tau)\left(\begin{array}{c} Y_2\\ Z_2\end{array}\right)(\tau,\xi),\end{equation}
with
$$M(\tau)=\frac{a^2}{4\tau^2\alpha^2}
\left(\begin{array}{cc} -1 & e^{-2i\Psi(\tau)}\\e^{2i\Psi(\tau)} & -1\end{array}\right),$$
and that $(Y_2,Z_2)(\tau,\xi)$ has a limit $(Y^+,Z^+)(\xi)$ as $\tau$ goes to infinity. Moreover, we obtained
$$\lim_{t\rightarrow\infty}\left(\widehat{\omega(t)}(\xi)-2e^{-it\xi^2}e^{i\frac{a^2}{2}\log t}\widehat{u_+}(\xi)\right)=0,$$
with
\begin{equation}
\label{z^+}
Z^+(\xi)=\frac 12\,e^{-ia^2\log|\xi|}\hat{u_+}(\xi),\quad Y^+(\xi)=\frac 12\,e^{ia^2\log|\xi|}\overline{\hat{u_+}}(-\xi).
\end{equation}
For $J(t)\omega(t)$ we shall use all this information in order to prove the following result.

\begin{prop}\label{linJ}
The solution $\omega$ of the linear equation \eqref{lin} satisfies for $t\xi^2\geq 2a^2$ that
$$|\widehat {J(t)\omega(t)}(\xi)|\leq C(a)\frac 1\xi(|\hat u_+(\xi)|+|\hat u_+(-\xi)|)$$
 and for  $t\xi^2\leq 2a^2$ that
$$|\widehat {J(t)\omega(t)}(\xi)|\leq C(a)\frac 1\xi(|\hat u_+(\xi)|+|\hat u_+(-\xi)|)+C(a,\delta)\frac{1}{\xi^{1+\delta}}(|\hat u_+(\xi)|+|\hat u_+(-\xi)|).$$
Therefore
$$\|J(t)\omega(t)\|_{L^2}\leq C(a)\|u_+\|_{\dot H^{-1}}+C(a,\delta)\|u_+\|_{\dot H^{-1-\delta}}.$$
\end{prop}
\begin{proof}
Following the first lines of the proof of Proposition A.1 in \cite{BV2} (using the notations therein), the same way we have obtained for $4a^2\leq \tau$
$$|Y_2(\tau,\xi)|^2+|Z_2(\tau,\xi)|^2\leq C(|\hat u_+(\xi)|^2+|\hat u_+(-\xi)|^2)$$
we also obtain
$$|\partial_\xi Y_2(\tau,\xi)|^2+|\partial_\xi Z_2(\tau,\xi)|^2\leq C(|\partial_\xi(e^{-ia^2\log\xi^2}\hat u_+(\xi))|^2+|\partial_\xi(e^{-ia^2\log\xi^2}\hat u_+(-\xi))|^2).$$
For $4a^2\leq t\xi^2$ we use \eqref{descromega} to get
$$\widehat{J(t)w(t)}(\xi)=(i\partial_\xi-2t\xi)\hat{w}(t,\xi)=(i\partial_\xi-2t\xi)\left (\left(1-\alpha\right)e^{i\Phi}Y_2+\left(1+\alpha\right)e^{-i\Phi}Z_2\right)(t\xi^2,\xi).$$
For any function $f$ we compute by using $\partial_\tau\Phi=\alpha$ and $\partial_\tau \alpha=\frac{a^2}{2\alpha\tau^2},$
\begin{equation}\label{formulas}(i\partial_\xi-2t\xi)\left(\left(1-\alpha(t\xi^2)\right)e^{i\Psi(t\xi^2)}f(\xi)\right)=i(1-\alpha)e^{i\Psi}\partial_\xi f-\frac{2a^2}{\xi}e^{i\Psi}f-\frac{ia^2}{\alpha t\xi^3}e^{i\Psi}f,\end{equation}
$$(i\partial_\xi-2t\xi)\left(\left(1+\alpha(t\xi^2)\right)e^{-i\Psi(t\xi^2)}f(\xi)\right)=i(1+\alpha)e^{-i\Psi}\partial_\xi f-\frac{2a^2}{\xi}e^{-i\Psi}f+\frac{ia^2}{\alpha t\xi^3}e^{-i\Psi}f.$$
In particular
$$\widehat{J(t)w(t)}(\xi)=i\left(1-\alpha\right) e^{i\Psi}\partial_2Y_2+i2t\xi\left(1-\alpha\right) e^{i\Psi}(M_{11}Y_2+M_{12}Z_2)-\frac{2a^2}{\xi}e^{i\Psi}Y_2-\frac{ia^2}{\alpha t\xi^3}e^{i\Psi}Y_2$$
$$+i\left(1+\alpha\right) e^{-i\Psi}\partial_2Z_2+i2t\xi\left(1+\alpha\right) e^{-i\Psi}(M_{21}Y_2+M_{22}Z_2)-\frac{2a^2}{\xi}e^{-i\Psi}Z_2+\frac{ia^2}{\alpha t\xi^3}e^{-i\Psi}Z_2.$$
By using the exact expression of $M$ we obtain that
\begin{equation}\label{link}
\widehat{J(t)w(t)}(\xi)=-\frac{2a^2}{\xi}\left(e^{i\Psi}Y_2+e^{-i\Psi}Z_2\right)+i\left(1-\alpha\right)e^{i\Psi}\partial_2Y_2+i\left(1+\alpha\right)e^{-i\Psi}\partial_2Z_2.
\end{equation}
In conclusion in the present region $2a^2\leq t\xi^2$ we have 
\begin{equation}\label{diagest}\sup_{2a^2\leq t\xi^2} |\hat v(t,\xi)|\leq \frac{C}{\xi}\,(|\hat u_+(\xi)|+|\hat u_+(-\xi)|)\end{equation}
$$+C(|\partial_\xi(e^{-ia^2\log\xi^2}\hat u_+(\xi))|^2+|\partial_\xi(e^{-ia^2\log\xi^2}\hat u_+(-\xi))|^2)$$
$$\leq \frac{C}{\xi}\,(|\hat u_+(\xi)|+|\hat u_+(-\xi)|).$$
so in particular
$$\| \hat v(t)\|_{L^2(\xi^2\geq\frac{2a^2}{t})}\leq C\sqrt{t}\|u_+\|_{L^2}.$$

For the cases $t\xi^2\leq 2a^2$ we can use energy methods like in Proposition A.1 in  \cite{BV2}  to connect \footnote{It is also possible to derive backwards pointwise Fourier estimates in the spirit of \cite{BV3}.}the time $t_1=t$ to the time $t_2=\frac{2a^2}{\xi^2}$ which in turn connects to the asymptotic state via \eqref{diagest}, and it is in here that we encounters the $\dot H^{-1}$ space
$$|\hat v(t,\xi)|\leq C(a) | \hat v(\frac{2a^2}{\xi^2},\xi)|+|\hat v(\frac{2a^2}{\xi^2},-\xi)|.
$$
$$+C(a,\delta_1,\delta_2)\frac{(\frac{2a^2}{\xi^2})^{\frac 12+\delta_1+\delta_2}}{t^{\delta_1}}\left(|\widehat{w}(\frac{2a^2}{\xi^2},\xi)|+|\widehat{w}(\frac{2a^2}{\xi^2},-\xi)|\right).$$
Using \eqref{diagest} and Remark A.2. from \cite{BV2} we conclude that
$$|\hat v(t,\xi)|\leq C(a)\frac 1\xi(|\hat u_+(\xi)|+|\hat u_+(-\xi)|)$$

$$+C(a,\delta)\frac{1}{\xi^{1+\delta}}(|\hat u_+(\xi)|+|\hat u_+(-\xi)|),$$
and the lemma follows.\\
\end{proof}

In the next proposition we give a precise result about the asymptotic behavior of $\widehat{J(t)\omega(t)}(\xi)$ for $\xi\neq 0$ fixed. 

\begin{prop}\label{linJasymptotics}For $\xi\neq 0$ we have the pointwise behavior
$$\lim_{t\rightarrow\infty}\left(\widehat{J(t)\omega(t)}(\xi)-2ie^{-i\tilde\Psi(t\xi^2)}\partial_\xi Z^+(\xi)+\frac{2a^2}{\xi}(e^{i\tilde\Psi(t\xi^2)}Y^+(\xi)+e^{-i\tilde\Psi(t\xi^2)}Z^+(\xi))\right)=0,$$
where
$$\tilde\Psi(t\xi^2)=t\xi^2-\frac{a^2}{2}\log t\xi^2.$$
\end{prop}

\begin{proof}

We write
$$\widehat{J(t)w(t)}(\xi)=(i\partial_\xi-2t\xi)\left (\left(1-\alpha(t\xi^2)\right)e^{i\Phi(t\xi^2)}Y^+(\xi)+\left(1+\alpha(t\xi^2)\right)e^{-i\Phi(t\xi^2)}Z^+(\xi)\right)$$
$$-(i\partial_\xi-2t\xi)\left (\left(1-\alpha\right)e^{i\Phi}\int_{\cdot}^\infty M_{11}Y_2+M_{12}Z_2+\left(1+\alpha\right)e^{-i\Phi}\int_{\cdot}^\infty M_{21}Y_2+M_{22}Z_2\right)(t\xi^2).$$
We denote $I_1$ the first term and $I_2$ the second. 
Since the entries of $M(\tau)$ are upper-bounded by $\frac{C(a)}{\tau^2}$ it follows from \eqref{formulas} that
$$|I_2(t,\xi)|\leq\frac{C(u_+)}{t\xi^3},$$
so this term is negligible for the pointwise asympotics in time. Using again \eqref{formulas} we obtain that 
$$|I_1(t,\xi)-i(1-\alpha)e^{i\Psi}\partial_\xi Y^+-i(1+\alpha)e^{-i\Psi}\partial_\xi Z^++\frac{2a^2}{\xi}(e^{i\Psi}Y^++e^{-i\Psi}Z^+)|\leq\frac{C(u_+)}{t\xi^3}.$$
We have $\alpha(\tau)=\sqrt{1-\frac{a^2}{\tau}}$, so 
$$|I_1(t,\xi)-2ie^{-i\Psi}\partial_\xi Z^++\frac{2a^2}{\xi}(e^{i\Psi}Y^++e^{-i\Psi}Z^+)|\leq\frac{C(u_+)}{t\xi^3},$$
and the statement of the proposition is proved, since $\tilde\Psi(\tau)=\Psi(\tau)+o(\frac 1\tau)$. 
\end{proof}

Using \eqref{z^+} and the above proposition we conclude that in order $J(t)\omega(t)$ to be uniformly in $L^2$ it is necessary that the asymptotic state $\hat u_+(\xi)$ has a null zero Fourier mode. This property is not preserved for solutions of the non-linear equation \eqref{NLS}, and therefore weighted spaces do not seem to be the right functional setting for  Theorem \ref{ivp} and Theorem \ref{cont}.


\begin{thebibliography}{99999}

\bibitem{ArHa} R.J. Arms and F.R. Hama, 
Localized-induction concept on a curved vortex and motion of an elliptic vortex
ring. 
{\em Phys. Fluids} {8} (1965), 553--560.

\bibitem{BV1}  V.~Banica and L.~Vega, 
On the stability of a singular vortex dynamics. 
{\it Comm. Math. Phys.} {286} (2009), 593--627.

\bibitem{BV2}  V.~Banica and L.~Vega, 
Scattering for 1D cubic NLS and singular vortex dynamics. 
{\it J. Eur. Math. Soc.} {14} (2012), 209--253.

\bibitem{BV3}  V.~Banica and L.~Vega, 
Stability of the self-similar dynamics of a vortex filament. 
{\it to appear in Arch. Ration. Mech. Anal.}. 

\bibitem{Bu} T. F. Buttke,
A numerical study of superfluid turbulence in the Self Induction Approximation. 
 {\it J. of Compt. Physics} {76} (1988), 301--326

\bibitem{Cal} A. Calini and T. Ivey, 
Stability of Small-amplitude Torus Knot Solutions of the Localized Induction Approximation. 
{\it J. Phys. A: Math. Theor.} {44} (2011) 335204. 

\bibitem{DaR} L. S. Da Rios,  
On the motion of an unbounded fluid with a vortex filament of any shape.  
{\em Rend. Circ. Mat. Palermo} {22} (1906), 117.

\bibitem{FuMi}
Y. Fukumoto and T. Miyazaki,  
Three dimensional distorsions at a vortex filament with axial velocity. 
{\em J. Fluid Mech.} {222} (1991), 396--416. 

\bibitem{GRV} S.~Guti\'errez, J.~Rivas and L.~Vega,  
Formation of singularities and self-similar vortex motion under the localized induction approximation. 
{\it Comm. Part. Diff. Eq.} {28} (2003), 927--968.

\bibitem{Ha}
H.~Hasimoto,
A soliton in a vortex filament. 
 {\it J. Fluid Mech.} {51} (1972), 477--485.

\bibitem{HGCV}
F.~de~la~Hoz, C.~Garc\'{i}a-Cervera and L.~Vega,
A numerical study of the self-similar solutions of the Schr\"odinger Map. 
{\it SIAM J. Appl. Math.} {70} (2009), 1047--1077. 

\bibitem{JeSm1}
R.~L.~Jerrard and D. Smets,
On Schr\"odinger maps from $T^1$ to $S^2$.
{\em Ann. Sci. \'Ec. Norm. Sup\'er.} {45} (2012), 637--680.

\bibitem{JeSm2}
R.~L.~Jerrard and D.~Smets,
On the motion of a curve by its binormal curvature.
{\em arXiv:1109.5483.}

\bibitem{Ko} N.~Koiso,
Vortex filament equation and semilinear Schr\"odinger equation. 
 {\em Nonlinear Waves, Hokkaido University Technical Report Series in Mathematics}
{43} (1996), 221--226.

\bibitem{Laf} S. Lafortune, 
Stability of solitons on vortex filaments.
{\em Phys. Lett. A} {377} (2013), 766--769.

\bibitem{LD} M.~Lakshmanan and M.~Daniel,
On the evolution of higher dimensional Heisenberg continuum spin systems.
{\em  Physica A} {107} (1981), 533--552.

\bibitem{LRT} M.~Lakshmanan, T.~W.~Ruijgrok, and C.~J.~Thompson,
On the the dynamics of a  continuum spin system. 
{\em Physica A} {84} (1976), 577--590.

\bibitem{Mag}
F. Maggioni, S. Z. Alamri, C. F. Barenghi, and R. L. Ricca, 
Velocity, energy and helicity of vortex knots and unknots.  
{\it Phys. Rev. E} {82} (2010), 26309--26317.

\bibitem{PMQ}
C.~S.~Peskin and D.~M.~McQueen, 
Mechanical equilibrium determines the fractal fiber architecture of aortic heart valve leaflets. 
{\it Am. J. Physiol. (Heart Circ. Physiol. 35)}  {266} (1994), H319--H328.

\bibitem{Ri1}  R.~L. Ricca, 
The contributions of Da Rios and Levi-Civita to asymptotic potential theory and vortex filament dynamics. 
 {\em Fluid Dynam. Res.} {18} (1996), 245--268. 

\bibitem{Ri2} R.~L. Ricca, 
Rediscovery of Da Rios equations. 
{\em Nature} {352} (1991), 561-562.

\end{thebibliography}
\end{document}